\documentclass[11pt]{article}
\usepackage[utf8]{inputenc}
\usepackage[babel]{microtype}
\usepackage[english]{babel}
\usepackage{fullpage}
\usepackage{comment}

\usepackage{amsmath,amssymb,amsthm, mathrsfs, mathtools,bbm,bm}
\usepackage[usenames,dvipsnames]{xcolor}
\usepackage{multirow}
\usepackage{tikz-cd,tikz}
\usetikzlibrary{arrows.meta, positioning}
\usetikzlibrary{decorations.pathreplacing, calligraphy}
\usepackage{enumitem}
\usepackage{adjustbox}
    \allowdisplaybreaks
\usepackage{graphicx} 
\usepackage{tikz}
\usepackage{subcaption}
\usepackage{float}
\usepackage{tikz,pgfplots}
\usetikzlibrary{arrows.meta}
\usetikzlibrary{decorations.pathreplacing,calligraphy}
\usetikzlibrary{positioning}

\setlist[enumerate,1]{label={(\roman*)}}
\setlist[enumerate,2]{label={(\alph*)}}
\setlist[enumerate,3]{label={(\arabic*)}}
\tikzcdset{scale cd/.style={every label/.append style={scale=#1},
    cells={nodes={scale=#1}}}}
\usepackage[hypertexnames=false]{hyperref}
\hypersetup{
    colorlinks,
    linkcolor={red!60!black},
    citecolor={green!60!black},
    urlcolor={blue!60!black}
}

\newcommand{\shi}{\mathcal{HS}}
\newcommand{\SSS}{\mathfrak{S}}
\newcommand{\lin}{\mathcal{H}}

\newcommand{\R}{\mathbb{R}}
\newcommand{\RR}{\mathcal{R}}
\newcommand{\idt}{\mathcal{T}}
\newcommand{\Z}{\mathbb{Z}}
\newcommand{\sg}{\mathfrak S}

\newcommand{\N}{\mathbb{N}}
\newcommand{\cl}[1]{\mathcal{#1}}

\DeclareMathOperator{\st}{\,|\,}

\DeclareMathOperator{\des}{des}
\DeclareMathOperator{\DES}{DES}
\DeclareMathOperator{\pw}{pw}

\DeclareMathOperator{\cyc}{cyc}
\DeclareMathOperator{\asc}{asc}
\DeclareMathOperator{\ASC}{ASC}
\DeclareMathOperator{\ID}{ID}
\DeclareMathOperator{\nbc}{NBC}
\DeclareMathOperator{\ch}{ch}
\DeclareMathOperator{\rk}{rk}

\newtheorem{theorem}{Theorem}[section]
\newtheorem{proposition}[theorem]{Proposition}
\newtheorem{corollary}[theorem]{Corollary}
\newtheorem{lemma}[theorem]{Lemma}

\newtheorem{question}[theorem]{Question}

\theoremstyle{definition}

\newtheorem{example}{Example}
\numberwithin{equation}{section}

\title{Homogenized Graphical Shi Arrangements and Deformed Dumont Permutations}
\author{
	Sauvik Poddar\\
    Department of Mathematics,\\
    Presidency University Kolkata,\\
    Kolkata 700073,\\
    West Bengal, India\\
	Email: \texttt{sauvikpoddar1997@gmail.com}
	\and
	Rutuja Sawant\thanks{The author is partially supported by a grant from the Infosys Foundation.}\\ Department of Mathematics,\\ Chennai Mathematical Institute,\\ Chennai 603103,\\ Tamil Nadu, India\\
    Email: \texttt{rutuja@cmi.ac.in}\\
    \and
	Umesh Shankar\\
	Department of Computer Science and Automation,\\ Indian Institute of Science Bengaluru,\\ Bengaluru 560012,\\ Karnataka, India\\ Email: \texttt{umeshshankar@outlook.com}\\
	}
\date{\today}

\begin{document}
\maketitle

\begin{abstract}
We introduce the homogenized graphical Shi arrangement associated with a simple undirected graph $G$, which serves as a broad generalization of classical deformations of the braid arrangement, including the Shi and homogenized Linial arrangements studied by Lazar and Wachs. We demonstrate that the intersection lattices of certain homogenized graphical Shi arrangements are isomorphic to the bond lattices of naturally associated graphs. For a distinguished family of graphs, by employing non-broken-circuit (NBC) techniques, we obtain explicit combinatorial interpretations of the arrangement's M\"obius function and characteristic polynomial. We achieve this by introducing generalisations of previously studied combinatorial objects: $\mathcal{R}$-deformed increasing-decreasing ($\mathcal{R}$-DID) forests and $\mathcal{R}$-D-permutations, establishing bijections between them to interpret the coefficients of the characteristic polynomial in terms of $\mathcal{R}$-D-permutations with prescribed number of cycles. Furthermore, we explore refinements of $\mathcal{R}$-D-permutations by their starting letters. Finally, we also resolve an open conjecture posed by Deutsch, Kitaev, and Remmel concerning the equidistribution of specific parity-constrained descent and ascent statistics.
\end{abstract}
\textbf{\small{}Keywords:}{\small{} hyperplane arrangement, characteristic polynomial, Dumont permutations, Linial arrangement, graphical Shi arrangements, Genocchi numbers,
increasing-decreasing forests.  }{\let\thefootnote\relax}


\section{Introduction}
The \emph{braid arrangement} (or \emph{type A Coxeter arrangement}) $\mathcal{A}_{n-1}$ is the arrangement of the hyperplanes given by 
\begin{equation}
    x_i-x_j=0,\hspace{1cm}  \ 1\le i<j\le n.
\end{equation}
The arrangement has $n!$ regions as the hyperplanes divide $\mathbb{R}^n$ into open cones of the form 
\begin{equation*}
    \{(x_1,\dots,x_n)\in \mathbb{R}^n : x_{\sigma(1)}<x_{\sigma(2)}<\cdots<x_{\sigma(n)} \},
\end{equation*}
where $\sigma$ is a permutation in the symmetric group $\SSS_n$. 

Zaslavsky \cite{zaslavsky-facing-arrangments} gives the number of regions of any real hyperplane arrangement $\mathcal{A}$ in terms of the M{\"o}bius function of its intersection lattice $\mathcal{L}(\mathcal{A})$. The \emph{characteristic polynomial} of a finite, ranked poset $P$ of length $\ell$, with minimum element $\hat{0}$ is defined to be
\begin{equation}
    \chi_{P}(t)\coloneqq \sum_{x\in P} \mu_{P} (\hat{0},x) t^{\ell - \rk(x)}
\end{equation}
where $\mu_P$ is the M{\"o}bius function of $P$ and $\rk(x)$ is the rank of the element $x$.

Zaslavsky shows that the number of regions of $\mathcal{A}$, $r(\mathcal{A})$, is 
\begin{equation}\label{Zaslavsky-region-formula}
    r(\mathcal{A})=(-1)^\ell \chi_{\mathcal{L}(\mathcal{A})}(-1). 
\end{equation}
Using the fact that every interval of the intersection poset of an arrangement is a geometric lattice, we have 
\begin{equation}
    \chi_\mathcal{A}(t)=\sum_{k=1}^n (-1)^{n-k}c_k t^{k}  
\end{equation}
where $c_k$ are non-negative integers for all $0\le k\le n$ (see, Stanley's notes \cite{stanley-hyperplane-notes}). When the regions of a hyperplane arrangement are in bijective correspondence with a certain combinatorial object, we could ask if there is a `statistic' on the combinatorial object whose distribution is given by the $c_k$'s. 

It is known and easy to see that the lattice of intersections of the braid arrangement $\mathcal{A}_{n-1}$ is isomorphic to the lattice $\Pi_n$ of partitions of the set $[n]$. It is also well known that the characteristic polynomial of $\Pi_n$ is given by 
\begin{equation}
    \chi_{\Pi_n}(t)=\sum_{k=1}^n s(n,k)t^{k-1},
\end{equation}
where $s(n,k)$ is the Stirling number of the first kind, which is equal to $(-1)^{n-k}$ times the number of permutations in $\SSS_n$ with exactly $k$ cycles (see, Stanley's notes \cite{stanley-hyperplane-notes}). 

Many deformations and generalisations of the braid arrangements are often studied in the literature, and in many cases, the number of regions is known. The regions of deformations of the braid arrangement often bijectively correspond to well-studied combinatorial objects (see, Stanley's notes \cite{stanley-hyperplane-notes}). One of the deformations of the braid arrangements is the \emph{Shi arrangement}. It is given by the hyperplanes in $\mathbb{R}^n$ defined by
\begin{eqnarray}\label{hyp: shi}
x_i-x_j=0,& \hspace{1cm} 1\le i<j\le n,\\
x_i-x_j=1,&  \hspace{1cm} 1\le i<j\le n.
\end{eqnarray}
By the work of Stanley \cite{stanley-hyperplane-interval-orders}, and Athanasiadis and Linusson \cite{athana-linusson-bij-shi-regions}, we know that the regions of the Shi arrangement corresponds bijectively to the number of parking functions on $[n]$.

Another hyperplane arrangement related to the braid arrangement is the \emph{Linial arrangement}. The Linial arrangement in $\mathbb{R}^n$ is defined by 
\begin{equation}
    x_i-x_j=1, \hspace{1cm} 1\le i\le j\le n.
\end{equation}
Postnikov and Stanley \cite{stanley-post-coxeter} show that the number of regions of this arrangement is the number of alternating trees on the nodes $[n+1]$, where a tree is \emph{alternating} if each vertex is either greater than all its neighbours or smaller than all its neighbours.

Hetyei \cite{hetyei-acyclic-tourney} introduced the \emph{homogenised Linial arrangement}, which is the hyperplane arrangement in
\begin{equation*}
    \{(x_1,\dots,x_n,y_1,\dots,y_{n-1})\in \mathbb{R}: x_i\in \mathbb{R} \ \forall i\in [n], \mbox{ and } y_i\in \mathbb{R}\ \forall i\in[n-1] \}=\mathbb{R}^{2n-1}
\end{equation*}
given by
\begin{equation}\label{eq:HLA}
    \lin_{2n-3} \coloneqq \{ x_i - x_j = y_i : 1\le i< j \le n \}.
\end{equation}
Hetyei proves that $r(\lin_{2n-1})=h_n$, where $h_n$ is the median Genocchi number. The intersection lattice $\mathcal{L}(\lin_{2n-1})$ and its characteristic polynomial were studied by Lazar and Wachs \cite{lazar-wachs-hom-linial}. Observing that the intersection lattice of this arrangement is isomorphic to the bond lattice of a certain bipartite graph, they obtain an interpretation of the number of regions of the homogenised Linial arrangement, the median Genocchi numbers, as the number of \emph{D-permutations}. The D-permutations are permutations $\sigma$ of $[2n]$ such that $i\le \sigma(i)$ if $i$ is odd, and $i\ge \sigma(i)$ whenever $i$ is even. 

They obtain the following interpretation of the coefficients of the characteristic polynomial of $\lin_{2n-1}$.
\begin{equation}
    \chi_{\lin_{2n-1}}(t)=\sum_{k=1}^{2n} s_D(2n,k)t^{k-1},
\end{equation}
where $(-1)^ks_D(2n,k)$ is equal to the number of D-permutations on $[2n]$ with exactly $k$ cycles.

Let $G=(V,E)$ be a simple, undirected graph such that $V=[n]$. In this work, we study hyperplane arrangements of the following form.
\begin{eqnarray}
    x_i-x_j=0,& \hspace{1cm} i<j,\ (i,j)\in E(G)\\
    x_i-x_j=y_i,& \hspace{1cm} 1\le i< j\le n.
\end{eqnarray}
We call this hyperplane arrangement \emph{homogenised graphical Shi arrangement} and denote it by $\lin_{2n-3}(G)$ .
When the graph $G$ is the complete graph $K_n$, then we get the \emph{homogenised Shi arrangement} in $\mathbb{R}^{2n-1}$ defined by 
\begin{eqnarray}
    \shi_{2n-1}\coloneqq \{x_i-x_j=0: 1\le i<j\le n \} \cup \{x_i-x_j=y_i: 1\le i<j\le n\}.
\end{eqnarray}
Intersecting this arrangment by $y_1=\cdots=y_{n-1}=1$ gives us the Shi arrangement described by Equations \ref{hyp: shi}. Also, when $G$ is edgeless graph on $n$ vertices, then we get the {\em homogenised Linial arrangement} defined in Equation \ref{eq:HLA}.

We now summarize the main results of this paper section-wise. The main goal of Sections~\ref{sec3} and~\ref{sec4} is to generalize the combinatorial structure underlying the homogenized Linial arrangement to a broader family of hyperplane arrangements, namely the homogenized graphical Shi arrangements. Motivated by the work of Lazar and Wachs on the intersection lattice of the homogenized Linial arrangement, in Section~\ref{sec3} we study the homogenized graphical Shi arrangement associated with an arbitrary graph $G$. We show that its intersection lattice is isomorphic to the bond lattice of a naturally associated graph. This identification enables us to study the M\"obius function and characteristic polynomial using non-broken-circuit techniques, and also yields a combinatorial description of the lattice elements in terms of suitable parity conditions on the blocks of set partitions.

In Section~\ref{sec4}, we specialize to a distinguished family of graphs with are a union of a clique and an independent set, whose associated bond lattices admit a richer combinatorial structure. To encode the corresponding NBC forests combinatorially, we introduce the notion of $\RR$-deformed increasing-decreasing forests ($\RR$-DID forests), where $\RR$ corresponds to the vertex set of the clique. These forests generalize the increasing-decreasing forests arising in the classical homogenized Linial arrangement setting. Our main result in Section \ref{sec4} is the following. 

\begin{theorem}\label{thm: char-poly-DID} 
Let $\mathcal{R} = \{1,3,\ldots,m\}$, where $m$ is odd integer. For all $\pi \in \Pi_{\Gamma_{2n-1}^{m}}$, we have that $(-1)^{|\pi|} \mu(\hat 0,\pi)$ equals the number of $\mathcal{R}$-DID forests on $[2n-1]$ whose trees have nodes sets equal to the blocks of $\pi$. Consequently,
$  -\mu_ {\Pi_{\Gamma_{2n-1}^{m}}}(\hat 0, \hat 1) $ is equal to the number of $\mathcal{R}$-DID trees on $[2n-1]$ and $$ \chi_{\Pi_{\Gamma_{2n-1}^{m}}}(t) = \sum_{k=1}^{2n-1} (-1)^{k} |\mathcal F_{2n-1,k}| t^{k-1},$$
where $\mathcal F_{2n-1,k}$ is the set of $\mathcal{R}$-DID forests on $[2n-1]$ with exactly $k$ trees.
\end{theorem}

We then define a corresponding class of Dumont-like permutations, called $\RR$-D-permutations, and construct bijections between these objects. As a consequence of Theorem \ref{thm: char-poly-DID}, we obtain the following combinatorial interpretations for the characteristic polynomial and its coefficients. 

\begin{theorem} \label{them: coeff to cycles}
Let $A$ be a finite subset of $\Z_{>0}$.  For  all $\pi \in  \Pi_{\Gamma_{A}^{\RR}}$,
$$(-1)^{|\pi|}\mu_{\Pi_{\Gamma_{A}}^{\RR}}(\hat{0},\pi) = |\{\sigma \in \mathcal{D}_{A}^{\RR}: \cyc(\sigma) = \pi\} |.$$ 
Consequently, \begin{equation} \label{chardumeq}  \chi_{\mathcal L\left(\Pi_{\Gamma_A}^{\RR}\right)}(t) = 
\sum_{k=1}^{2n} s_{\RR D}(A,k) t^{k-1},\end{equation}
where $(-1)^{k}s_{\RR D}(A,k)$ is equal to  the number of $\RR$-D-permutations on $A$ with exactly $k$ cycles.
\end{theorem} 

In Section \ref{sec5}, we obtain the characteristic polynomial of the hyperplane arrangement $\lin_{2n-3}(K_n)$ and give another combinatorial interpretation of the characteristic polynomial in terms of the cycles of pairs of permutations. 

In Section \ref{Sec6}, we define and study refinements of the permutation sets $\mathcal{D}_{2n+1,0}$ and $\mathcal{D}_{2n+1,n+1}$ with respect to the starting letter. We give a simple bijection between the D-permutations on $[2n]$ (the permutations in $\mathcal{D}_{2n+1,n+1}$ with last letter removed) defined by Lazar and Wachs \cite{lazar-wachs-hom-linial} and the Dumont derangements on $[2n+2]$. This lets us use an existing recurrence for the Dumont derangements to enumerate the D-permutations with fixed starting letter. The problem of enumerating the number of regions for other choices of $\RR$ is still open.

In Section \ref{sec: conj}, we translate the set $\mathcal{D}_{n,0}$ bijectively into the set of permutations whose descent tops are all odd. These permutations have been studied before by Kitaev and Remmel \cite{kitaev-remmel-parity}.
The number of descents of $p\in \SSS_n$, $\des(p)$, is the cardinality of the set $\DES(p):=\{i\in[n-1]: p_i>p_{i+1}\}$. For an index $i\in \DES(p)$, we call $p_i$ the descent top and $p_{i+1}$ the descent bottom. Similarly, we can define the ascent statistic, denoted by $\asc$. The number of ascents of $p\in \SSS_n$, $\asc(p)$, is the cardinality of the set $\ASC(p):=\{i\in[n-1]:p_i<p_{i+1}\}$. 
Define the following permutation statistics. 
\begin{enumerate}
    \item $S_{10}(p):=|\{i\in[n-1]: p_i>p_{i+1} \mbox{ and }  p_{i}\in \mathbb{O}\}|$,
    \item $S_{12}(p):=|\{i\in[n-1]:p_{i},p_{i+1}\in \mathbb{O}\}|$,
    \item $S_{17}(p)$ is the largest $i$ such that the elements $1,\dots ,i$ appear from left to right in $p$.
\end{enumerate}
 Therefore, the set of permutations whose descent tops are all odd are the permutations where $\mathrm{des}(\pi)=S_{10}(\pi)$ where $\mathrm{des}(\pi)$ is the number of descents in $\pi$.

Deutsch, Kitaev and Remmel \cite[Conjecture $10$]{equi-desc-ap-pv-deutsch-kitaev} conjectured that the the triples $(S_{10},S_{12}, S_{17})$ and $(S_{12},S_{10},S_{17})$ are equidistributed. In fact, they give a refined conjecture based on the following $3$ permutation statistics.
\begin{enumerate}
    \item $T_1(p):=|\{ i\in[n-1]: p_i>p_{i+1},\ p_i\in \mathbb{O}, \ p_{i+1}\in\mathbb{E} \}|,$ 
    \item $T_2(p):=|\{ i\in[n-1]: p_i<p_{i+1} \mbox{ and } p_i,p_{i+1}\in \mathbb{O} \}|,$
    \item $T_3(p):=|\{ i\in[n-1]: p_i>p_{i+1} \mbox{ and } p_i,p_{i+1}\in \mathbb{O} \}|,$
\end{enumerate}

We give an involution that takes $T_1$ to $T_2$ and $T_2$ to $T_1$ while fixing $T_3, S_{17}$, which results in the following.

\begin{theorem}{\cite[Conjecture $11$]{equi-desc-ap-pv-deutsch-kitaev}}\label{thm: main}
    The quadruples $(T_1,T_2,T_3,S_{17})$ and $(T_2,T_1,T_3,S_{17})$ are equidistributed over $\SSS_n$ for all $n\in \N$.
\end{theorem}

This also gives a proof of \cite[Conjecture $10$]{equi-desc-ap-pv-deutsch-kitaev} which we state as a corollary below.
\begin{corollary}{\cite[Conjecture $10$]{equi-desc-ap-pv-deutsch-kitaev}}\label{cor: S10_S12}
    The triples of statistics $(S_{10}, S_{12}, S_{17})$ and $(S_{12}, S_{10}, S_{17})$ are equidistributed over $\SSS_n$ for all $n \in \N$.
\end{corollary}

\section{Preliminaries}\label{sec:pre}

\subsection{Hyperplane Arrangements}
Let $k$ be a field, where $k=\mathbb{R}$ or $\mathbb{C}$. A {\em hyperplane arrangement} $\mathcal{A}\subseteq k^n$ is a finite collection of affine subspaces of codimension $1$ in $k^n$.

The set of all nonempty intersections of hyperplanes in $\mathcal{A}$, ordered by reverse inclusion, is called the {\em intersection poset} of $\mathcal{A}$, and is denoted by $\mathcal{L}(\mathcal{A})$. The ambient space $k^n$ itself is included as an element of $\mathcal{L}(\mathcal{A})$, viewed as the intersection over the empty collection of hyperplanes. When $\bigcap_{H\in \mathcal{A}} H \neq \varnothing$,
the intersection poset is a geometric lattice; otherwise, it is a geometric semilattice.

Now suppose that $\mathcal{A}$ is a real hyperplane arrangement in $\mathbb{R}^n$. In this case, the complement $\mathbb{R}^n \setminus \mathcal{A}$
is disconnected. The number of connected components of this complement is called the {\em number of regions} of the arrangement and is denoted by $r(\mathcal{A})$. Remarkably, this quantity depends only on the intersection poset $\mathcal{L}(\mathcal{A})$, as expressed by Zaslavsky's formula (\ref{Zaslavsky-region-formula}).

\subsection{Bond lattice of a graph}

In this subsection, we refer to Lazar and Wachs \cite{lazar-wachs-hom-linial}. We recall some of the definitions required.

Let $G=(V,E)$ be a graph. Given a subset $B$ of $V$, let $G\big|_B$ denote the induced subgraph of $G$ with vertex set $B$. Let $\Pi_V$ denote the lattice of partitions of the set $V$ ordered as usual by reverse refinement. The \emph{bond lattice of $G$} is the induced subposet $\Pi_G$ of $\Pi_V$ consisting of partitions $\pi=B_1\,|\dots|\,B_k$ such that $G\big|_{B_i}$ is connected for all $i$.

\begin{figure}[H]
\centering

\begin{subfigure}{0.42\textwidth}
\centering
\begin{tikzpicture}[scale=0.7]
\tikzstyle{node} = [circle, draw, fill=blue!20, minimum size=1mm]
    \node[node] (node1) at (0,0) {2};
    \node[node] (node2) at (2,2) {3};
    \node[node] (node3) at (4,0) {4};
    \node[node] (node4) at (2,-2) {1};
    \draw (node1)--(node2);
    \draw (node1)--(node3);
    \draw (node2)--(node3);
    \draw (node2)--(node4);
\end{tikzpicture}
\caption{$G$}
\end{subfigure}
\hspace{-0.4cm}
\begin{subfigure}{0.42\textwidth}
\centering
\begin{tikzpicture}[scale=0.6]
  \node (one) at (0,4.5) {$1234$};
  \node (a1) at (-4,2) {$123|4$};
  \node (a2) at (-2,2) {$13|24$};
  \node (a3) at (2,2) {$1|234$};
  \node (a4) at (4,2) {$134|2$};
  \node (a) at (-4,-1) {$1|2|34$};
  \node (b) at (-2,-1) {$1|23|4$};
  \node (c) at (2,-1) {$13|2|4$};
  \node (d) at (4,-1) {$1|3|24$};
  \node (zero) at (0,-4) {$1|2|3|4$};
  \draw (a) -- (a3);
  \draw (a) -- (a4);
  \draw (b) -- (a1);
  \draw (b) -- (a3);
  \draw (c) -- (a1);
  \draw (c) -- (a2);
  \draw (c) -- (a4);
  \draw (d) -- (a2);
  \draw (d) -- (a3);
  \draw (one) -- (a1);
  \draw (one) -- (a2);
  \draw (one) -- (a3);
  \draw (one) -- (a4);
  \draw (zero) -- (a);
  \draw (zero) -- (b);
  \draw (zero) -- (c);
  \draw (zero) -- (d);
\end{tikzpicture}
\caption{$\Pi_G$}
\end{subfigure}
\caption{A graph $G$ and its bond lattice $\Pi_G$.}\label{fig:bond lattice}
\end{figure}
Whitney's formula \cite{whitney-logical-exp} for the chromatic polynomial $\ch_G(t)$ of $G$ is given by 
\begin{equation}\label{bond-lattice-chromatic-poly}
    \ch_G(t)=t\chi_{\Pi_G}(t).
\end{equation}
To compute the M{\"o}bius function of the bond lattice of a graph $G=(V,E)$, we fixed a total ordering on $E$. A subset $S\subset E$ is called a \emph{broken circuit} if it is the edge set of a cycle of $G$ with its smallest edge (with respect to this ordering) removed. If $S$ does not contain a broken circuit, we say that $S$ is a \emph{non-broken circuit set} or \emph{$\nbc$} set. The graph $(V,S)$ is necessarily a forest and we call $(V,S)$ an \emph{$\nbc$ forest} of $G$.

Given any $S\subset E$, let $\pi_S$ be the partition of $V$ whose blocks are the vertex sets of the connected components of the graph $(V,S)$. For $\pi\in \Pi_G$, 
\begin{equation}
    (-1)^{\rk(\pi)}\mu(\hat 0, \pi)=\# \{ \nbc \mbox{ forests } (V,S) \mbox{ of } G: \pi_S=\pi \}.
\end{equation}

\subsection{Deformed increasing-decreasing trees}


We recall the definition of an increasing-decreasing ($\ID$) tree defined by Lazar and Wachs \cite{lazar-wachs-hom-linial}.  Let $T$ be a tree whose nodes are in $\mathbb{Z}_{>0}$. We say that $T$ is \emph{increasing-decreasing} if, when $T$ is rooted at its largest vertex (or smallest vertex), each internal vertex $v$ satisfies
\begin{enumerate}
    \item if $v$ is odd, then $v$ is less than all its descendants and all its children are even,
    \item if $v$ is even, then $v$ is greater than all its descendants and all its children are odd.
\end{enumerate}
An \emph{increasing-decreasing} (\emph{$\ID$}) forest is a forest whose components are all increasing-decreasing trees.
We now introduce a variant, called the $\mathcal{R}$-deformed increasing-decreasing ($\mathcal{R}$-DID) tree. Let $\mathcal{R}:=\{1,3,\dots, 2j-1\}\cap V(T)$ for some $j\in \mathbb{N}$.
A tree $T$ is said to be a \emph{$\mathcal{R}$-deformed increasing-decreasing tree} if, when $T$ is rooted at its smallest vertex, each internal vertex $v$ satisfies:
\begin{enumerate}
    \item if $v$ is odd and belongs to $\mathcal{R}$, then $v$ is less than all of its descendants, and each child of $v$ is either even or an odd vertex contained in $\mathcal{R}$,
    \item if $v$ is odd and does not belong to $\mathcal{R}$, then $v$ is less than all of its descendants, and all of its children are even,
    \item if $v$ is even, then $v$ is greater than all of its descendants, and all of its children are odd nodes not contained in $\mathcal{R}$.
\end{enumerate}
An \emph{$\mathcal{R}$-deformed increasing-decreasing forest} is a forest whose components are all $\mathcal{R}$-deformed increasing-decreasing trees. An example of $\mathcal{R}$-DID forest on $[12]$, where $\mathcal{R}=\lbrace{1,3,5,7}\rbrace$ is given in the figure below.

\begin{figure}[h]
\centering

\begin{tikzpicture}[scale=0.7]
\tikzstyle{node} = [circle, draw, fill=blue!20, minimum size=1mm]
    \node[node] (node1) at (0,0) {2};
    \node[node] (node2) at (2,0) {3};
    \node[node] (node3) at (2,2) {1};
    \node[node] (node4) at (1,-2) {4};
    \node[node] (node5) at (3,-2) {5};
    \node[node] (node6) at (4,0) {6};

    \draw (node1)--(node3);
    \draw (node3)--(node6);
    \draw (node2)--(node3);
    \draw (node2)--(node4);
    \draw (node2)--(node5);

    \node[node] (node7) at (6,0) {8};
    \node[node] (node8) at (8,0) {10};
    \node[node] (node9) at (7,2) {7};

    \draw (node7)--(node9);
    \draw (node8)--(node9);

    \node[node] (node10) at (10,0) {12};
    \node[node] (node11) at (10,2) {9};

    \draw (node10)--(node11);

    \node[node] (node12) at (12,2) {11};
\end{tikzpicture}
\caption{An $\mathcal{R}$-DID forest on $12$ nodes.}\label{R-DID forest}
\end{figure}

\section{Intersection lattices of $\lin_{2n-3}(G)$}\label{sec3}
In this section, we show that for any graph $G$, $\mathcal{L}(\mathcal{H}_{2n-3}(G))$ is isomorphic to the bond lattice of a certain graph. 

Let $S_1\coloneqq \{1,3,\dots, 2n-1\}$ and $S_2\coloneqq \{2,4,\dots, 2n-2\}$. Define $\Gamma_G$ to be the graph with the following edges on the vertex set $S_1\sqcup S_2$. For $1\le i,j\le n$, $(i,j)$ is an edge in $G$ if and only if $(2n-2i+1,2n-2j+1)$ is an edge in $\Gamma_G$. Also, $(2i-1,2j)\in E(\Gamma_G)$ whenever $1\le i\le j\le n-1$.

A graph $G$ on $[5]$ and its $\Gamma_G$ is shown below.

\begin{figure}[H]
\centering
\begin{tikzpicture}[scale=0.8]
\tikzstyle{node} = [circle, draw, fill=blue!20, minimum size=1mm]

    \node[node] (node1) at (2.5,0) {2};
    \node[node] (node2) at (5.5,0) {5};
    \node[node] (node3) at (4,1) {1};
    \node[node] (node4) at (3,-2) {3};
    \node[node] (node5) at (5,-2) {4};

    \draw (node1)--(node2);
    \draw (node1)--(node3);
    \draw (node2)--(node3);
    \draw (node1)--(node4);
    \draw (node4)--(node5);
    \draw (node2)--(node5);

    \node at (4, -3.5) {(a) $G$};

    \node[node] (r_node1) at (8.5,0) {7};
    \node[node] (r_node2) at (11.5,0) {1};
    \node[node] (r_node3) at (10,1) {9};
    \node[node] (r_node4) at (9,-2) {5};
    \node[node] (r_node5) at (11,-2) {3};

    \node[node] (r_node6) at (15,0.5) {2};
    \node[node] (r_node7) at (15,-0.5) {4};
    \node[node] (r_node8) at (15,-1.5) {6};
    \node[node] (r_node9) at (15,-2.5) {8};

    \draw (r_node1)--(r_node2);
    \draw (r_node2)--(r_node5);
    \draw (r_node2)--(r_node3);
    \draw (r_node1)--(r_node3);
    \draw (r_node4)--(r_node5);
    \draw (r_node1)--(r_node4);

    \draw (r_node2)--(r_node6);
    \draw (r_node2)--(r_node7);
    \draw (r_node2)--(r_node8);
    \draw (r_node2)--(r_node9);
    \draw (r_node5)--(r_node7);
    \draw (r_node5)--(r_node8);
    \draw (r_node5)--(r_node9);
    \draw (r_node4) to [bend right=25] (r_node8);
    \draw (r_node4) to [bend right=25] (r_node9);
    \draw (r_node1)--(r_node9);

    \node at (11.5, -3.5) {(b) $\Gamma_G$};

\end{tikzpicture}
\caption{A graph $G$ and its $\Gamma_G$.}\label{G-and-Gamma-G}
\end{figure}

\normalsize

\begin{theorem}\label{thm: poset isomorphism for a general graph}
    The poset isomorphisms $\mathcal{L}(\lin_{2n-3}(G)) \cong \mathcal{L}(\mathcal{A}_{\Gamma_G}) \cong \Pi_{\Gamma_G}$ holds for all $n\ge 1$.
\end{theorem}
\begin{proof}
Let $(e_1,e_2, \dots,e_{2n-1}) $ be the standard basis for $\R^{2n-1}$. For $1\le i \le j \le n-1$, let
$$ H_{i,j} = \{v \in \R^{2n-1} : (e_{i}-e_{n+i} - e_{j+1}) \cdot v = 0 \} $$
and 
$$K_{i,j} = \{v \in \R^{2n-1} : (e_{2i-1}-e_{2j}) \cdot v = 0 \} .$$

Also, if $\{i,j\} \in E(G)$, we define
$$H_{i,j}' = \{v \in \R^{2n-1} : (e_{i}-e_{j}) \cdot v = 0 \}$$
and 
$$K_{i,j}' = \{v \in \R^{2n-1} : (e_{2n-2i+1}-e_{2n-2j+1}) \cdot v = 0 \} .$$
Clearly $\{H_{i,j} : 1\le i \le j \le n-1\} \cup \{H_{i,j}' : \{i,j\} \in E(G)\}$ is precisely the arrangement $\mathcal H_{2n-3}(G)$. Note that 
$\mathcal K_{2n-1}:=\{ K_{i,j} : 1\le i \le j \le n-1 \} \cup \{K_{i,j}' : \{i,j\} \in E(G)\}$ is the graphical arrangement $\mathcal A_{\Gamma_G}$. It is well known and easy to see that $$\mathcal L(\mathcal A_{\Gamma_G} ) \cong \Pi_{\Gamma_{G}}.$$  
We will prove the result by producing a vector space isomorphism $\psi: \R^{2n-1} \to \R^{2n-1}$ that takes $\mathcal A_{\Gamma_G}$ to $\mathcal{H}_{2n-3}(G)$.  Indeed,  such a map will induce an  isomorphism from $\mathcal{L}(\mathcal A_{\Gamma_G})$ to $\mathcal{L}(\mathcal{H}_{2n-3}(G)) $.  
  
First consider the linear operator $\phi:\R^{2n-1} \to \R^{2n-1}$ defined   on the standard basis by letting $\phi(e_{2i-1}) = e_{n-(i-1)}$ and $\phi(e_{2i}) = e_{n-i}-e_{2n-i}$ for all $i \in [n-1]$ and letting $\phi(e_{2n-1}) = e_1$ .  Let $A$ be the matrix of $\phi$ with respect to the standard basis.  One can easily check that $|\det A| = 1$.  Now let $\psi :\R^{2n-1} \to \R^{2n-1}$ be the linear operator whose matrix with respect to the standard basis is the transpose of $A^{-1}$.  Clearly $\psi$ is an isomorphism.  
 
We claim that $\psi$ takes the hyperplane $K_{i,j}$ to the hyperplane $H_{i,j}$ for all $1\le i \le j\le n-1$ and $K_{i,j}'$ to the hyperplane $H_{i,j}'$ for $\{i,j\} \in E(G)$. To prove the claim, let $v \in  K_{i,j}$, so $(e_{2i-1}-e_{2j}) \cdot v = 0$.   We have $$\phi (e_{2i-1}-e_{2j}) \cdot \psi(v) = (e_{n-(i-1)} - e_{n-j} + e_{2n-j}) \cdot \psi(v).$$  
We also have \begin{align*} \phi (e_{2i-1}-e_{2j}) \cdot \psi(v) &=(A (e_{2i-1}-e_{2j})) \cdot ((A^{-1})^T v) \\ &= (A^{-1}A  (e_{2i-1}-e_{2j})) \cdot  v \\ &=  (e_{2i-1}-e_{2j}) \cdot  v \\  &= 0 .\end{align*}
It follows that $(e_{n-(i-1)} - e_{n-j} + e_{2n-j}) \cdot \psi(v) = 0$.  Hence $\psi(v)$ is in $H_{n-j,n-i}$. 
Now, let $v \in K_{i,j}'$, so $(e_{2n-2i+1}-e_{2n-2j+1}) \cdot v = 0$. We have 
\begin{align*}\phi (e_{2n-2i+1}-e_{2n-2j+1}) \cdot \psi(v) &= \phi (e_{2(n-i+1)-1} -e_{2(n-j+1)-1}) \cdot \psi(v) \\
 &= (e_{i} - e_{j}) \cdot \psi(v) \\
\end{align*} 
We also have \begin{align*} \phi (e_{2i-1}-e_{2j-1}) \cdot \psi(v) &=(A (e_{2i-1}-e_{2j-1})) \cdot ((A^{-1})^T v) \\ &= (A^{-1}A  (e_{2i-1}-e_{2j-1})) \cdot  v \\ &=  (e_{2i-1}-e_{2j-1}) \cdot  v \\  &= 0 .\end{align*}
It follows that $(e_{i} - e_{j}) \cdot \psi(v) = 0$.  Hence $\psi(v)$ is in $K_{i,j}$,
which proves the claim.  It follows from the claim that $\psi$ takes $\mathcal K_{2n-1}$ to $\mathcal{H}_{2n-3}(G)$ as desired.
\end{proof}

\section{$\mathcal{R}$-DID forests and Dumont-like permutations}\label{sec4}

Consider $n,k \in \Z_{>0}$, where \(m\) is odd and satisfies \(m \leq 2n-1\). Let \(G\) be the graph defined by  
\[
V(G) = [n], 
\qquad 
E(G) = \bigl\{\{i,j\} : n- \lfloor \tfrac{m}{2} \rfloor \leq i < j \leq n \bigr\}.
\]
For this type of graph \(G\), we define \(\Gamma_G\) as follows. Let  
\[
S_1 \coloneqq \{1,3,\dots,2n-1\}, 
\qquad 
S_2 \coloneqq \{2,4,\dots,2n-2\}.
\]
Then \(\Gamma_{2n-1}^{m}\) is the graph with vertex set  
\[
V(\Gamma_{2n-1}^{m}) = S_1 \sqcup S_2
\]
and edge set  
\[
E(\Gamma_{2n-1}^{m}) =
\bigl\{ \{2i-1,\,2j\} : 1 \leq i \leq j \leq n-1 \bigr\}
\;\cup\;
\bigl\{ \{i,\,j\} : i,j \in \{1,3,5,\ldots,m\}, \; i<j \bigr\},
\]
where \(m\) is an odd integer.

\medskip

\noindent
\textbf{Example.}  
The graph \(\Gamma_{7}^{3}\) is obtained by taking \(n=4\) and \(m=3\).

\medskip
\noindent
More generally, for any finite subset $V$ of $\mathbb{Z}_{> 0}$ and any odd integer $m$, let $\Gamma_{V}^{m}$ be a graph on $V$ with an edge set  
\[
\begin{aligned}
E={}&\{\{u,v\} : u,v\in V,\; u<v,\; \text{$u$ is odd and $v$ is even}\} \\
&{}\cup
\{\{u,v\} : u,v\in V,\; u<v\le m,\; \text{$u$ and $v$ are odd}\}.
\end{aligned}
\]

\medskip

From Theorem \ref{thm: poset isomorphism for a general graph}, the poset isomorphisms $\mathcal{L}(\lin_{2n-3}(G)) \cong \mathcal{L}(\mathcal{A}_{\Gamma_{2n-1}^{m}}) \cong \Pi_{\Gamma_{2n-1}^{m}}$ holds for all $n\ge 1$. Hence, to study $\mathcal{L}(\lin_{2n-3}(G))$, it suffices to study $\Pi_{\Gamma_{2n-1}^{m}}$.

\medskip

\begin{proposition}\label{prop:BondLatticeStructure}
The bond lattice $\Pi_{\Gamma_{2n-1}^{m}}$ is the induced subposet of $\Pi_{2n-1}$ consisting of those partitions $\pi$ in which every non-singleton block has either  
\begin{enumerate}
    \item an odd minimum and an even maximum, or  
    \item an odd minimum and an odd maximum, both belonging to $\{1,3,5,\ldots,m\}$.
\end{enumerate}
\end{proposition}

\begin{proof}
Suppose that $\pi \in \Pi_{\Gamma_{2n-1}^{m}}$, and let $B$ be a non-singleton block of $\pi$ with $u = \min B$.  
Since $\Gamma_{2n-1}^{m}[B]$ is connected, there exists some $v \in B$ such that $\{u,v\}$ is an edge of $\Gamma_{2n-1}^{m}$.  
By the definition of $\Gamma_{2n-1}^{m}$, and since $u < v$, it follows that $u$ must be odd.  

Now let $v = \max B$.  
Again, because $\Gamma_{2n-1}^{m}[B]$ is connected, there exists some $u \in B$ with $\{u,v\}$ an edge of $\Gamma_{2n-1}^{m}$.  
By the definition of $\Gamma_{2n-1}^{m}$, and since $u < v$, it follows that $v$ must be either even or an odd number belonging to $\{1,3,\ldots,m\}$.  

Conversely, suppose that we have a partition $\pi = B_1|\cdots|B_k \in \Pi_{2n-1}$ such that each non-singleton block $B_i$ has either an odd minimum and an even maximum, or an odd minimum and an odd maximum belonging to $\{1,3,5,\ldots,m\}$. We show that $\Gamma_{2n-1}^{m}[B_i]$ is connected.  

\begin{enumerate}
    \item Let $m_i = \min B_i$ (odd) and $M_i = \max B_i$ (even). By the definition of $\Gamma_{2n-1}^{m}$, each odd vertex in $B_i$ is adjacent to $M_i$, and each even vertex in $B_i$ is adjacent to $m_i$. In particular, $m_i$ and $M_i$ are adjacent, so $\Gamma_{2n-1}^{m}[B_i]$ is connected.  

    \item Let $m_i = \min B_i$ (odd) and $M_i = \max B_i$ (odd, with $M_i \in \{1,3,\ldots,m\}$). By the definition of $\Gamma_{2n-1}^{m}$, each odd vertex in $B_i$ is adjacent to $M_i$, and each even vertex in $B_i$ is adjacent to $m_i$. In particular, $m_i$ and $M_i$ are adjacent, so $\Gamma_{2n-1}^{m}[B_i]$ is connected.  
\end{enumerate}
\end{proof}

To prove Theorem ~\ref{thm: char-poly-DID}, we will require the following lemma, which describes the lexicographic order on the edges of the graph under consideration.  

\begin{lemma}[Blockwise odd--odd then odd--even lexicographic order]\label{lem:Blockwise odd-odd then odd-even lexicographic order}
Let \(n\ge 2\) and
\[
O=\{1,3,5,\dots,2n-1\},\qquad
\mathcal P=\{(a,b)\in\mathbb Z^2:\ a\in O,\ a<b\le 2n-1\}.
\]
For each fixed \(a\in O\) set \(B_a=\{b\in\mathbb Z:\ a<b\le 2n-1\}\).
Define a relation $\prec_a$ on $B_a$ by declaring, for $b_1,b_2\in B_a$:
\begin{enumerate}
  \item if $b_1,b_2$ are odd then $b_1\prec_a b_2 \iff b_1<b_2$ (usual order);
  \item if $b_1$ is odd and $b_2$ is even then $b_1\prec_a b_2$;
  \item if $b_1,b_2$ are even then $b_1\prec_a b_2 \iff b_1>b_2$ (reverse usual order on evens).
\end{enumerate}
Now define a relation $\prec$ on $\mathcal P$ by
\[
(a,b)\prec(c,d)\quad\Longleftrightarrow\quad
\big(a<c\big)\ \text{or}\ \big(a=c\ \text{and}\ b\prec_a d\big),
\]
where $<$ on the first coordinate is the usual order on $O$ (i.e. $1<3<5<\cdots<2n-1$).

Then $\prec$ is a total (linear) order on $\mathcal P$.  Moreover, with respect to $\prec$ the pairs appear blockwise by the first coordinate $a$ (increasing odd $a$), and for each fixed $a$ the second coordinates appear with all odd $b$ in increasing order followed by all even $b$ in decreasing order.  In particular,
\[
(1,3)\prec(1,5)\prec\cdots\prec(1,2n-1)\prec(1,2n-2)\prec(1,2n-4)\prec\cdots\prec(1,2)
\prec(3,5)\prec(3,7)\prec\cdots
\]
\end{lemma}

\begin{proof}[Sketch of proof]
For each fixed \(a\), the relation \(\prec_a\) is a total order on \(B_a\): \(B_a\) is partitioned into odd and even elements; odd elements are ordered in increasing order, even elements are ordered in decreasing order, and every odd is declared smaller than every even. Hence \(\prec_a\) is total, antisymmetric, and transitive.

The ordering \(\prec\) on \(\mathcal P\) is the lexicographic combination of the usual total order on the first coordinate \(a\) (restricted to odd \(a\)) with the total order \(\prec_a\) on the second coordinate. Lexicographic combination of total orders is again a total order, so \(\prec\) is a linear order. The explicit description of the listing follows immediately from the definition of \(\prec_a\) and the fact that blocks are ordered by increasing \(a\). That shows, for instance, \((1,2n-1)\prec(1,2n-2)\) and also, every \((1,\ast)\) precedes every \((3,\ast)\).
\end{proof}

We are now ready to prove Theorem \ref{thm: char-poly-DID}.

\begin{proof}[Proof of Theorem \ref{thm: char-poly-DID}]
    We prove this result by showing that under the appropriate ordering of the edges of $\Gamma_{2n-1}^{m}$, the NBC forests of $\Gamma_{2n-1}^{m}$ are precisely the $\mathcal{R}$-DID forests on $[2n-1]$. For the total order on the edges of $\Gamma_{2n-1}^{m}$ we use the order in Lemma \ref{lem:Blockwise odd-odd then odd-even lexicographic order}. In particular, by associating edges $\{i,j\}$ with the ordered pair $(i,j)$ we have the following:
\[
(1,3)\prec(1,5)\prec\cdots\prec(1,2n-1)\prec(1,2n-2)\prec(1,2n-4)\prec\cdots\prec(1,2)
\prec(3,5)\prec(3,7)\prec\cdots
\] 
Since adding an edge between two different components of a forest can never complete a cycle, a forest is an NBC forest of $\Gamma_{2n-1}^{m}$ if and only if trees are NBCs trees of the subgraph of $\Gamma_{2n-1}^{m}$ induced by the node set of the trees. Hence, it suffices to prove that the NBC trees of $\Gamma_{2n-1}^{m}$ are precisely the $\mathcal{R}$-DID trees on $[2n-1]$. 

Let $T$ be an $\mathcal{R}$-DID tree on $[2n-1]$. Suppose that $e = \{u,v\}$ is an edge of $\Gamma_{2n-1}^{m}$ that is not in $T$. If we add $e$ to $T$, we create a unique cycle $\rho$, which contains $e$. To show that $T$ is NBC, it suffices to find an edge of $\rho$ that strictly precedes $e$ in the order $\prec$. 

We will complete this in two parts:

\begin{enumerate}
    \item Suppose that $e = \{2i-1,2j\}$. Root $T$ at its smallest node and let $h$ be the youngest common ancestor of $2i-1$ and $2j$ (that is, the ancestor farthest from the root). Clearly, $h$ is a node of $\rho$.   

    Suppose $h$ is odd. Recall that either $h \in \mathcal{R}$ or $h \notin \mathcal{R}$. Then $h \leq 2i-1$, since $T$ is an $\mathcal{R}$-DID tree.  If $h$ is the parent of $2j$, then $h$ cannot be $2i-1$, because $\{2i-1,2j\}$ is not an edge of $T$. Therefore, $h < 2i-1$, which implies that $\{h,2j\} \prec \{2i-1,2j\} = e$.  Thus, $\{h,2j\}$ is the edge of $\rho$ that we seek.
If $h$ is not the parent of $2j$, then let $2k$ or $2k-1$ be the child of $h$ that is a proper ancestor of $2j$.  We thus have either $2k > 2j$ or $2k-1 < 2j$, which implies either $\{h, 2k\} \prec \{2i-1, 2j\} = e$ or $\{h, 2k-1\} \prec \{2i-1, 2j\} = e$.  Thus, in this case, $\{h, 2k\}$ or $\{h,2k-1\}$ is the edge of $\rho$ that we seek.

The case that $h$ is even is handled analogously.

    \item Suppose that $e = \{i,j\}$, where $i <j \; \text{and} \; \{i,j\} \subseteq \mathcal{R}$. Root $T$ at its smallest node and let $h$ be the youngest common ancestor of $i$ and $j$ (that is, the ancestor farthest from the root). Clearly, $h$ is a node of $\rho$, and $h$ is an odd node belonging to $\mathcal{R}$ with $h \le i$, since $T$ is an $\mathcal{R}$-DID tree.   

    If $h$ is the parent of $j$, then $h$ cannot be $i$, because $\{i,j\}$ is not an edge of $T$. Therefore, $h < i$, which implies that  $\{h,j\} \prec \{i,j\} = e$.  Thus, $\{h,j\}$ is the edge of $\rho$ that we seek.
If $h$ is not the parent of $j$, then let $k \in \mathcal{R}$ be the child of $h$ that is a proper ancestor of $j$.  We, thus, have $k < j$, which implies $\{h, k\} \prec \{i,j\} = e$.  Thus, in this case, $\{h, k\}$ is the edge of $\rho$ that we seek.
\end{enumerate}
We can now conclude that $T$ is an NBC tree of $\Gamma_{2n-1}^{m}$. 

\medskip
Conversely, let $T$ be an NBC tree of $\Gamma_{2n-1}^{m}$. Suppose that $T$ is not an $\mathcal{R}$-DID. Then $T$ rooted at its smallest node has a node $x$ with one of the following properties:

\begin{enumerate}
    \item an even node $x$ with a child $y \in \mathcal{R}$.
    \item an odd node $x$ with a child $y \in \mathcal{R}$ with $x > y$.
    \item an even node $x$ with a grandchild $y$ satisfying $x < y$.
    \item an odd node $x$ with a grandchild $y$ satisfying $x > y$.
\end{enumerate}
To complete the proof, we split the argument into four cases:

\medskip
\textbf{Case I:} An even $x$ with a child $y \in \mathcal{R}$. 

Since $T$ is rooted at its smallest node and $x$ is not the root, since the smallest node of $T$ must be odd, let $u$ be the parent of $x$. Thus, $u, x, y$ form a path in $T$. 
Observe that $x > u$ and $x > y$. 
Note that $u \notin \mathcal{R}$, because if $u \in \mathcal{R}$, then $\{u, y\}$ is an edge in $\Gamma_{2n-1}^{m}$, 
and by the order on the edge set, we have $\{u, y\} \prec \{u, x\}$ and $\{u, y\} \prec \{x, y\}$. 
This implies that the subgraph of $T$ induced by $\{u, x, y\}$ is a broken circuit, 
which contradicts the hypothesis that $T$ is an NBC tree. 
Hence, $u \notin \mathcal{R}$.

Observe that the smallest node, being an odd node, must belong to $\mathcal{R}$. 
Moreover, every odd node from $\mathcal{R}^c$ is larger than every node in $\mathcal{R}$, and $y \in \mathcal{R}$.

Let $v$ be the youngest ancestor of $y$ that belongs to $\mathcal{R}$. 
It follows that $\{v, y\}$ is an edge of $\Gamma_{2n-1}^{m}$. 
Let $\rho$ be the path induced by the node set $\{v, u, x, y\}$ in $T$. 
Clearly, $\{v, y\} \prec \{a, b\}$ for every $\{a, b\} \in E(\rho)$. 
This implies that $\rho$ is a broken circuit, which contradicts the hypothesis that $T$ is an NBC tree.

\medskip
\textbf{Case II:} An odd $x \in \mathcal{R}$ with a child $y \in \mathcal{R}$ such that $x > y$.

Since $T$ is rooted at its smallest node and $x$ is not the root (as $x > y$), 
let $u$ be the parent of $x$, which is an odd node belonging to $\mathcal{R}$ by Case~I. 
Thus, $u, x, y$ form a path in $T$ with $u < x$ and $y < x$, by the minimality of $x$. 
Since $u, y \in \mathcal{R}$, it follows that $\{u, y\}$ is an edge of $\Gamma_{2n-1}^{m}$. 
Observe that, in the total order of the edge set, $\{u, y\} \prec \{u, x\}$ and $\{u, y\} \prec \{x, y\}$. 
This implies that the subgraph of $T$ induced by $\{u, x, y\}$ is a broken circuit, 
which contradicts the hypothesis that $T$ is an NBC tree. 

\medskip
\textbf{Case III:} An even $x$ with a grandchild $y$ satisfying $x < y$.

Since $T$ is rooted at its smallest node and $x$ is not the root, the smallest node of $T$ must be odd. Let $u$ be the parent of $x$ and let $v$ be the parent of $y$, which implies that $x$ is the parent of $v$.  Thus, $u,x,v,y$ form a path in $T$. 

We claim that $\{u,y\}$ is an edge of $\Gamma_{2n-1}^{m}$. As $x$ is an even node, $u$ must be an odd node with $x > u$. Thus, we have $u < x < y$. From Case I, we know that $v$ is an odd node not from $\mathcal{R}$, which implies that $y$ is an even node. Tt follows that $\{u,y\}$ is an edge in $\Gamma_{2n-1}^{m}$. In the total order on the edge set, $\{u,y\} \prec \{u,x\}$. By the minimality of the choice of $x$, since $u$ is closer to the root than $x$ is, we have $u < v$. Hence, $\{u,y\} \prec \{u,x\}, \{v,x\}, \{v,y\}$. This implies that the subgraph of $T$ induced by $\{u,x,v,y\}$ is a broken circuit which contradicts the hypothesis that $T$ is an NBC tree. 

\medskip
\textbf{Case IV:} An odd $x$ with a grandchild $y$
satisfying $x >y$.

Since $T$ is rooted at its smallest node, and $x$ is not the root (as $x > y$). Let $u$ be the parent of $x$ and let $v$ be the parent of $y$, which implies that $x$ is the parent of $v$.   Thus, $u,x,v,y$ form a path in $T$. Since $y$ is a grandchild of $x$, it has to be an odd node. 

We will complete this case using three subcases:

\medskip
\textbf{Sub-case I:} $x \notin \mathcal{R}$.

In this sub-case, from Case I, we have $u$ is an even node, $v$ is an even node and $y$ is an odd node which does not belong to $\mathcal{R}$. Thus, we have $u > x > y$. It follows that $\{u,y\}$ is an edge of $\Gamma_{2n-1}^{m}$. In the total order on the edge set, $\{u,y\} \prec \{u,x\}$. By the minimality of the choice of $x$, since $u$ is closer to the root than $x$ is, we have $u > v$. Hence, $\{u,y\} \prec \{u,x\}, \{v,x\}, \{v,y\}$. This implies that the subgraph of $T$ induced by $\{u,x,v,y\}$ is a broken circuit, which contradicts the hypothesis that $T$ is an NBC tree. 

\medskip
\textbf{Sub-case II:} $x \in \mathcal{R}$ and $v \in \mathcal{R}$.

In this sub-case, $u,v,x,y \in \mathcal{R}$. By Case II, we have $u < x < v < y$.

\medskip
\textbf{Sub-case III:} $x \in \mathcal{R}$ and $v$ is even node.

In this sub-case, $u,x \in \mathcal{R}$ and $y$ is an odd node which does not belong to $\mathcal{R}$. In $\Gamma_{2n-1}^{m}$, $x < y$. 

\medskip
Therefore, $T$ is $\mathcal{R}$-DID.
\end{proof}
\subsection{Dumont-like permutations}

We introduce a class of permutations analogous to Dumont permutations and give a bijection between these permutations and the $\RR$-DID forests, where $\mathcal{R} = \{1,3,\ldots,m\}$ for some odd integer $m$. 

\subsubsection{$\RR$-D-permutations} 
Let $A$ be a finite subset of $\mathbb{Z}_{>0}$. A permutation $\sigma \in S_A$ is an $\mathcal{R}$-D-permutation on $A$ if  
$i \le \sigma(i)$ whenever $i$ is odd and $i \notin \mathcal{R}$, and $i \ge \sigma(i)$ whenever $i$ is even.  
We write $\mathcal{RD}_A$ for the set of $\RR$-D-permutations on $A$ and $\mathcal{RDC}_A$ for the set of $\RR$-D-cycles on $A$. If $A = [n]$, we denote these by $\mathcal{D}_n^{\RR}$ and $\mathcal{DC}_n^{\RR}$.

\begin{example}
Let $A = [6]$ and $\RR = \{1,3\}$. Then examples of $\RR$-D-permutations are

\[
(3421)(65) \qquad (431)(65)(2) \qquad  (321)(65)(4) \qquad (21)(43)(65)
\]

The $\RR$-$\mathcal{D}$-cycles on $[6]$ are:

\[
(564321) \qquad (563421) \qquad (435621) \qquad (356421)
\]
\end{example}

\subsubsection{Bijection between $\RR$-DID trees and $\RR$-D-cycles}

For any tree $T$ (rooted or unrooted), let $|T|$ denote the number of nodes of $T$ and $V(T)$ denote the set of vertices(node). 
Recall that a \emph{plane tree} is a rooted tree in which the children of each node are linearly ordered. Let $T$ be a plane tree on a subset of $\Z_{>0}$. When $T$ is drawn in the plane, the children of each node $v$ are drawn from left to right according to their linear order.

Let $\idt$ be the set of $\RR$-DID trees.  For $T \in \idt$, 
let $\hat T$ be the plane tree obtained  by rooting $T$ at its smallest node (which is odd when $|T| > 1$) and fixing the following total ordering on the children of each node $v$ of $T$:
\begin{itemize}
    \item If $v$ is even, then arrange all its children (which are necessarily odd and not contained in $\RR$) in increasing order from left to right.

    \item If $v$ is odd and $v \notin \mathcal{R}$, then arrange all its children (which are necessarily even) in decreasing order from left to right.

    \item If $v$ is odd and $v \in \mathcal{R}$, then its children may include both even nodes and nodes in $\mathcal{R}$. Arrange them as follows:
    \begin{itemize}
        \item First, place the children that belong to $\mathcal{R}$ in increasing order from left to right.
        \item Then, place the even children in decreasing order from left to right.
    \end{itemize}
\end{itemize}

Also, if $V(T) \cap \RR = \emptyset$ then let $\tilde T$ be the plane tree obtained by rooting $T$ at its largest node (which is even when $|T| >1$) and fixing the above total ordering on the children of each node of $T$.

\begin{figure}[h]
\centering

\begin{subfigure}{0.42\textwidth}
\centering
\begin{tikzpicture}[scale=0.7]
\tikzstyle{node} = [circle, draw, fill=blue!20, minimum size=1mm]

    \node[node] (node1) at (0,0) {2};
    \node[node] (node2) at (2,0) {3};
    \node[node] (node3) at (2,2) {1};
    \node[node] (node4) at (1,-2) {4};
    \node[node] (node5) at (3,-2) {5};
    \node[node] (node6) at (4,0) {6};

    \draw (node1)--(node3);
    \draw (node3)--(node6);
    \draw (node2)--(node3);
    \draw (node2)--(node4);
    \draw (node2)--(node5);

\end{tikzpicture}
\caption{$T$}
\end{subfigure}
\hspace{-0.4cm}
\begin{subfigure}{0.42\textwidth}
\centering
\begin{tikzpicture}[scale=0.7]
\tikzstyle{node} = [circle, draw, fill=blue!20, minimum size=1mm]

    \node[node] (node1) at (0,0) {3};
    \node[node] (node2) at (2,0) {6};
    \node[node] (node3) at (2,2) {1};
    \node[node] (node4) at (1,-2) {5};
    \node[node] (node5) at (-1,-2) {4};
    \node[node] (node6) at (4,0) {2};

    \draw (node1)--(node3);
    \draw (node3)--(node6);
    \draw (node2)--(node3);
    \draw (node1)--(node4);
    \draw (node1)--(node5);

\end{tikzpicture}
\caption{$\widehat{T}$}
\end{subfigure}

\caption{An $\{1,3\}$-DID tree $T$ and its plane tree rooted at smallest node $\widehat{T}$.}\label{fig:PlaneTree}
\end{figure}

Now let $T$ be a plane  tree  on a subset of $\Z_{> 0}$.  When drawn in the plane the order is depicted from left to right.  For each vertex $v$, the order of its children determines  a  ``left to right" order of the subtrees rooted at the children.  If  $|T|>1$, let  $T_1,T_2,\dots,T_k$ be the subtrees rooted at the children of the root $r$ of $T$ ordered from ``left to right".  The   {\em postorder word} $\pw(T)$ of $T$ is defined recursively as the concatenation  $$\pw(T) = \pw(T_1) \cdot \pw(T_2)\cdot \,\, \cdots \,\,  \cdot \pw(T_k)\cdot  r,$$
if $|T| > 1$ and by $\pw(T)= r$ if $|T| = 1$. For the plane tree depicted in Figure~\ref{fig:PlaneTree}, 
\[
\operatorname{pw}(\widehat{T}) = 453621.
\]

Let $\mathcal W$ be the set of all words $w_1w_2\dots w_m$ with $m \ge 1$, whose letters are distinct positive integers, satisfying:
\begin{enumerate}
 \item For all $i \in [m-1]$, if $w_i$ is odd and $w_i \notin \RR$, then $w_i < w_{i+1}$.
 \item For all $i \in [m-1]$, if $w_i$ is even, then $w_i > w_{i+1}$.
 \item If $w_m$ is odd, then it is the smallest letter of $w$.
 \item If $w_m$ is even, then $w_i \notin \RR$ for all $i \in [m-1]$, and $w_m$ is the largest letter of $w$.
\end{enumerate}

\begin{lemma}
Let $T \in \idt$ with $n = |T|$. Then $\mathrm{pw}(\hat T) \in \mathcal{W}$, and if $V(T) \cap \RR = \emptyset$ then $\mathrm{pw}(\tilde T) \in \mathcal{W}$.
\end{lemma}
\begin{proof} Suppose $T \in \idt$. We prove the result for $\hat T$ and $\tilde T$ simultaneously by induction on $|T|$. The result is obviously true for $|T|=1$, so assume $|T| >1$.

Let
\[
w=\begin{cases}
\mathrm{pw}(\hat T), & \text{for the first statement},\\
\mathrm{pw}(\tilde T), & \text{when }V(T)\cap\mathcal R=\emptyset .
\end{cases}
\]
Write \(w=w_1\cdots w_n\).

If \(w=\mathrm{pw}(\hat T)\), then \(w_n\) is the root of \(\hat T\),
hence the smallest node of \(T\), which is odd when \(|T|>1\).

If \(w=\mathrm{pw}(\tilde T)\), then \(V(T)\cap\mathcal R=\emptyset\),
and \(w_n\) is the root of \(\tilde T\), hence the largest node of \(T\),
which is even when \(|T|>1\).
Therefore condition~(3) or condition~(4) in the definition of
\(\mathcal W\) holds accordingly.

Thus it remains to verify conditions~(1) and~(2) for all \(i<n\).
\begin{itemize}
    \item if $w_i$ is odd and $w_i \notin \mathcal{R}$, then $w_i < w_{i+1}$, and
    \item if $w_i$ is even, then $w_i > w_{i+1}$.
\end{itemize}

The arguments below depend only on the prescribed left-to-right
ordering of children, which is the same in both \(\hat T\) and
\(\tilde T\).

\medskip
\noindent
\noindent\textbf{Odd case:}  
Let $w_i \notin \mathcal{R}$ be odd node. In $\hat T$, the vertex $w_i$ is a child of some even vertex $x$. Since, $T$ is a $\RR$-DID tree, $x$ is larger than all of its descendants, and all the children of $x$ are odd and do not belong to $\mathcal{R}$.

Let
\[
\mathrm{Ch}(x) = \{x_1, x_2, \ldots, x_\ell\}
\quad \text{with} \quad
x_1 < x_2 < \cdots < x_\ell,
\]
where $\mathrm{Ch}(x)$ denotes the set of the children of $x$.

We consider two cases.

\medskip
\noindent\textbf{Case 1:} $w_i = x_\ell$.  
Then, in $\mathrm{pw}(\hat T)$, we have $w_{i+1} = x$. Since $x$ is larger than all its children,
\[
w_i = x_\ell < x = w_{i+1}.
\]

\medskip
\noindent\textbf{Case 2:} $w_i = x_j$ for some $j < \ell$.  
Then $w_{i+1}$ lies in the subtree rooted at $x_{j+1}$. Since $x_j < x_{j+1}$, and every descendant of $x_{j+1}$ is at least $x_{j+1}$, it follows that
\[
w_i = x_j < x_{j+1} \le w_{i+1}.
\]

Thus, in both cases, $w_i < w_{i+1}$.

\medskip
\noindent\textbf{Even case:}  
Now let $i < n$ and suppose that $w_i$ is even. In $\hat T$, the vertex $w_i$ is a child of some odd vertex $x$. Again as $T$ is $\RR$-DID tree, $x$ is smaller than all of its descendants.

We again consider two cases.

\medskip
\noindent\textbf{Case 1:} $x \in \mathcal{R}$.  
The children of $x$ consist of nodes in $\mathcal{R}$ and even nodes. Write
\[
\mathrm{Ch}(x) = V_1 \sqcup V_2,
\]
where $V_1 \subseteq \mathcal{R}$ and $V_2$ consists of all even children of $x$. By construction of $\hat T$, nodes in $V_1$ appear to the left of nodes in $V_2$. Let
\[
V_2 = \{x_1, x_2, \ldots, x_\ell\}
\quad \text{with} \quad
x_1 > x_2 > \cdots > x_\ell.
\]

If $w_i = x_\ell$, then $w_{i+1} = x$, and since $x$ is smaller than all its children,
\[
w_i = x_\ell > x = w_{i+1}.
\]

If $w_i = x_j$ for some $j < \ell$, then $w_{i+1}$ lies in the subtree rooted at $x_{j+1}$. Since $x_j > x_{j+1}$ and every descendant of $x_{j+1}$ is at most $x_{j+1}$, we obtain
\[
w_i = x_j > x_{j+1} \ge w_{i+1}.
\]

\medskip
\noindent\textbf{Case 2:} $x \notin \mathcal{R}$.  
In this case, all children of $x$ are even. The same argument as above applies, and we again obtain $w_i > w_{i+1}$.

\medskip
This completes the proof.
\end{proof}

Given a word $w = w_1 w_2 \cdots w_m$ over alphabet $\Z_{\ge 0}$ , we say that $w_j$ is a right-to-left minimum if $w_j < w_i$ for all $i > j$. Similarly we say that $w_j$ is a right-to-left maximum if $w_j > w_i$ for all
$i > j$. For example, the right-to-left minima of $45813267$ are $1, 2, 6, 7$ and right-to-left maxima are $7, 8$. We also let $w '$ denote $w$ with its last letter removed, that is, $w' = w_1 \cdots w_{m-1}$.

\begin{lemma}\label{lem:child}
Let $T\in\idt$ and let $u$ be a vertex of $T$.

\begin{itemize}
    \item If $u$ is even, then the children of $u$ in $\hat T$
    are exactly the right-to-left minima of
    $\pw(\hat T_u)$.
    
    \item If $u$ is odd and $u\notin\RR$, then the children of $u$
    in $\hat T$ are exactly the right-to-left maxima of
    $\pw(\hat T_u)$.
    
    \item If $u\in\RR$, write
    \[
    \pw(\hat T_u)=w_1w_2\cdots w_m u,
    \]
    and let $w_i$ be the first right-to-left letter in the word
    $w_1w_2\cdots w_m$ that belongs to $\RR$.
    Then:
    \begin{itemize}
        \item the children of $u$ belonging to $\RR$
        are exactly the right-to-left minima in
        $w_1\cdots w_i$ which belongs to $\RR$;

        \item the remaining children are exactly the
        right-to-left maxima in
        $w_{i+1}\cdots w_m$.
    \end{itemize}
\end{itemize}

The analogous statements also hold for $\tilde T$
whenever $V(T)\cap\RR=\emptyset$. Moreover, $\hat T_u$ denotes the rooted subtree of $\hat T$ with root $u$.
\end{lemma}

\begin{proof}
We prove the statements for $\hat T$. The proof for $\tilde T$
is completely analogous.

\medskip

First suppose that $u$ is an even vertex.
Then all children of $u$ are odd nodes that do not belong
to $\RR$, and they are arranged in increasing order from
left to right.
Let
\[
v_1<\cdots <v_k
\]
be the children of $u$ in $\hat T$.

For each $1\le j\le k$, write
\[
\pw(\hat T_{v_j})=\alpha_j v_j.
\]
Then
\[
\pw(\hat T_u)
=
\alpha_1 v_1
\alpha_2 v_2
\cdots
\alpha_k v_k u.
\]

Since each $v_j$ is odd and does not belong to $\RR$,
it is smaller than all of its descendants.
Hence, for $i<j$, the vertex $v_i$ is smaller than every
letter appearing in the subtree rooted at $v_j$.
Therefore, $v_i$ is smaller than every letter appearing
to its right in $\pw(\hat T_u)$, so each child of $u$
is a right-to-left minimum of $\pw(\hat T_u)$.

Conversely, suppose that $v$ is not a child of $u$.
Then $v$ belongs to the subtree rooted at some child
$v_j$, and $v\neq v_j$.
Since $v_j$ is odd and does not belong to $\RR$,
it is smaller than all of its descendants. Thus
\[
v_j<v.
\]
Moreover, $v_j$ appears to the right of $v$ in the
postorder word $\pw(\hat T_u)$.
Hence there exists a smaller letter to the right of $v$,
so $v$ is not a right-to-left minimum.
Therefore, the children of $u$ are exactly the
right-to-left minima of $\pw(\hat T_u)$.

\medskip

Now suppose that $u$ is odd and $u\notin\RR$.
Then all children of $u$ are even and are arranged in
decreasing order from left to right.
Let
\[
v_1>\cdots >v_k
\]
be the children of $u$ in $\hat T$.

For each $1\le j\le k$, write
\[
\pw(\hat T_{v_j})=\alpha_j v_j.
\]
Then
\[
\pw(\hat T_u)
=
\alpha_1 v_1
\alpha_2 v_2
\cdots
\alpha_k v_k u.
\]

Since each $v_j$ is even, it is larger than all of its
descendants.
Hence, for $i<j$, the vertex $v_i$ is larger than every
letter appearing in the subtree rooted at $v_j$.
Therefore, $v_i$ is larger than every letter appearing
to its right in $\pw(\hat T_u)$, so each child of $u$
is a right-to-left maximum of $\pw(\hat T_u)$.

Conversely, suppose that $v$ is not a child of $u$.
Then $v$ belongs to the subtree rooted at some child
$v_j$, and $v\neq v_j$.
Since $v_j$ is even, it is larger than all of its
descendants. Thus
\[
v_j>v.
\]
Moreover, $v_j$ appears to the right of $v$ in the
postorder word $\pw(\hat T_u)$.
Hence there exists a larger letter to the right of $v$,
so $v$ is not a right-to-left maximum.
Therefore, the children of $u$ are exactly the
right-to-left maxima of $\pw(\hat T_u)$.

\medskip

Finally, suppose that $u\in\RR$.
Then the children of $u$ consist of nodes in $\RR$
and even nodes.
In $\hat T$, the children belonging to $\RR$
are arranged in increasing order from left to right,
followed by the even children arranged in decreasing order.

Let
\[
v_1,\ldots,v_i
\]
be the children of $u$ belonging to $\RR$, and let
\[
v_{i+1},\ldots,v_k
\]
be the even children of $u$.
Then
\[
v_1<\cdots <v_i
\qquad\text{and}\qquad
v_{i+1}>\cdots >v_k.
\]

For each $1\le j\le k$, write
\[
\pw(\hat T_{v_j})=\alpha_j v_j.
\]
Then
\[
\pw(\hat T_u)
=
\alpha_1 v_1
\cdots
\alpha_i v_i
\alpha_{i+1} v_{i+1}
\cdots
\alpha_k v_k u.
\]

We first consider the children belonging to $\RR$.
Since $v_a<v_b$ for all $1 \le a<b\le i$, and every descendant
of $v_b$ is at least $v_b$, it follows that
$v_a$ is smaller than every letter appearing to its right
inside
\[
\alpha_1 v_1\cdots \alpha_i v_i.
\]
Hence each $v_a$ is a right-to-left minimum in this word.

Conversely, if a vertex appearing in
\[
\alpha_1 v_1\cdots \alpha_i v_i
\]
is not one of the children $v_1,\ldots,v_i$,
then it lies in the subtree rooted at some $v_b$ and is
strictly larger than $v_b$.
Since $v_b$ appears to its right in the postorder word,
that vertex cannot be a right-to-left minimum.
Therefore, the children of $u$ belonging to $\RR$
are exactly the right-to-left minima in
\[
\alpha_1 v_1\cdots \alpha_i v_i.
\]

Now consider the even children
\[
v_{i+1},\ldots,v_k.
\]
Since $v_a>v_b$ for $i+1\le a<b\le k$, and every
descendant of $v_b$ is at most $v_b$, it follows that
$v_a$ is larger than every letter appearing to its right
inside
\[
\alpha_{i+1} v_{i+1}\cdots \alpha_k v_k.
\]
Hence each $v_a$ is a right-to-left maximum in this word.

Conversely, if a vertex appearing in
\[
\alpha_{i+1} v_{i+1}\cdots \alpha_k v_k
\]
is not one of the children $v_{i+1},\ldots,v_k$,
then it lies in the subtree rooted at some even child
$v_b$ and is strictly smaller than $v_b$.
Since $v_b$ appears to its right in the postorder word,
that vertex cannot be a right-to-left maximum.
Therefore, the remaining children are exactly the
right-to-left maxima in
\[
\alpha_{i+1} v_{i+1}\cdots \alpha_k v_k.
\]

The proof for $\tilde T$ is completely analogous.
\end{proof}

\begin{lemma}\label{lem:BijWithW}
    The map $ \pw: \{\hat T : T \in \idt\} \cup \{\tilde T :  T \in \idt \; \text {and} \; V(T) \cap \RR = \emptyset \} \to \mathcal W $ is a bijection
\end{lemma}
\begin{proof}
To show that $\pw$ is a bijection, we construct a recursive map
\[
\gamma:\mathcal W\to
\{\hat T:T\in\idt\}
\cup
\{\tilde T:T\in\idt \text{ and }V(T)\cap\RR=\emptyset\}
\]
and show that it is inverse to $\pw$.

Let
\[
w=w_1\cdots w_m\in\mathcal W.
\]

If $m=1$, define $\gamma(w)$ to be the tree consisting of
the single vertex $w_1$.

Now suppose $m>1$.

\medskip

\noindent{\bf Case 1: $w_m$ is even.}

Then, by condition~(4), $w_m$ is the largest letter of $w$
and no letter of
\[
w'=w_1\cdots w_{m-1}
\]
belongs to $\RR$.

Let
\[
w_{j_1}<\cdots <w_{j_k}
\]
be the right-to-left minima of $w'$. Note that $j_k = m-1$.

For each $i=1,\ldots,k$, define
\[
\alpha_i(w)
=
w_{j_{i-1}+1}\cdots w_{j_i},
\qquad j_0=0.
\]

Then
\[
w'
=
\alpha_1(w)\cdots\alpha_k(w).
\]

Since each $w_{j_i}$ is a right-to-left minimum of $w'$,
it cannot be even, because every even letter $w_i$
satisfies
\[
w_i>w_{i+1},
\]
so $w_i$ has a smaller letter to its right.
Moreover, condition~(4) in the definition of $\mathcal W$
implies that no letter of $w'$ belongs to $\RR$.
Hence each $w_{j_i}$ is odd and does not belong to $\RR$.

Since for all $i$,  $\alpha_i(w)$  is a segment of $w \in \mathcal W$, it satisfies conditions (1) and (2) of the definition of $\mathcal W$.

Moreover, it is the smallest letter of $\alpha_i(w)$.
Hence each $\alpha_i(w)$ satisfies condition~(3) and condition (4) holds vacuously.
Therefore $\alpha_i(w) \in \mathcal W$.

Now, for each $i$, we can recursively apply $\gamma$ to  $\alpha_i(w)$ to obtain a plane tree  $ \gamma(\alpha_i(w))$  in $\{\tilde T : T \in \idt\}$.  Let $\gamma(w)$ be the plane tree constructed as follows: 

\begin{enumerate}
\item The root is $w_m$, which is even and largest.
\item The children of $w_m$ are the right-to-left minima $w_{j_1}< \cdots <w_{j_k}$, which are odd and does not belongs to $\RR$.
\item The subtree rooted at $w_{j_i}$ is $ \gamma(\alpha_i(w))$ for each $i$.
\end{enumerate}
Clearly $\gamma(w) \in  \{\tilde T : T \in \idt \; \text{and} \; V(T) \cap \RR = \emptyset\}$.

\medskip
\noindent{\bf Case 2: $w_m$ is odd and $w_m\notin\RR$.}

An analogous argument allows us to decompose $w = w_1 \cdots w_{m-1} \in {\mathcal W}$ into  segments $\alpha_i(w) \in \mathcal W$, whose last letter is the $i$-th  right--to-left maximum of $w$. We define  $\gamma(w)$ to be the plane tree constructed as follows: 
\begin{enumerate}
\item The root is $w_m$, which is odd.
\item The children of $w_m$ are the right-to-left maxima $w_{j_1}> \cdots > w_{j_k}$, which are even
\item The subtree rooted at $w_{j_i}$ is $\gamma(\alpha_i(w))$ for each $i$.
\end{enumerate} 
We have $\gamma(w) \in  \{\tilde T : T \in \idt\}$.

\medskip
\noindent{\bf Case 3: $w_m$ is odd and $w_m\in\RR$.}

Again, $w_m$ is the smallest letter of $w$.

Let
\[
w_i
\]
be the first right-to-left word of $w'$ that belongs to $\RR$.

By Lemma~\ref{lem:child}, the children of a node in $\RR$
consist first of the right-to-left minima of the word segment $w_1\ldots w_i$ belonging to $\RR$,
followed by the right-to-left maxima in the remaining suffix.

Let
\[
w_{j_1}<\cdots <w_{j_t}
\]
be the right-to-left minima in
\[
w_1\cdots w_i
\]
and let
\[
w_{j_{t+1}}> \cdots >w_{j_{t+s}}
\]
be the right-to-left maxima in
\[
w_{i+1}\cdots w_{m-1}.
\]

Note that, $j_r = i$ and $t+s = m-1$.
For convenience, relabel these nodes as
\[
v_1,\ldots,v_t,v_{t+1},\ldots,v_{t+s},
\]
where
\[
v_1<\cdots <v_t
\]
are the right-to-left minima belonging to $\RR$, and
\[
v_{t+1}>\cdots >v_{t+s}
\]
are the right-to-left maxima in
\[
w_{i+1}\cdots w_{m-1}.
\]

For each $i=1,\ldots,t+s$, define
\[
\alpha_i(w)
=
w_{j_{i-1}+1}\cdots v_i,
\qquad j_0=0.
\]

Then
\[
w'
=
\alpha_1(w)\cdots\alpha_{t+s}(w).
\]

Since for all $i$, the word $\alpha_i(w)$ is a segment of
$w\in\mathcal W$, it satisfies conditions~(1) and~(2)
in the definition of $\mathcal W$.

For $1\le i\le r$, the last letter of $\alpha_i(w)$ belongs
to $\RR$ and is the smallest letter of $\alpha_i(w)$,
since it is a right-to-left minimum.
Hence condition~(3) holds.

For $r+1\le i\le r+s$, the last letter of $\alpha_i(w)$
is even and largest in $\alpha_i(w)$, since it is a
right-to-left maximum.
Moreover, no preceding letter of $\alpha_i(w)$ belongs
to $\RR$, because all letters in $\RR$ occur in
\[
w_1\cdots w_i.
\]
Hence condition~(4) holds.

Therefore $\alpha_i(w)\in\mathcal W$ for all $i$.

Now, for each $i$, we recursively apply $\gamma$ to
$\alpha_i(w)$ to obtain a plane tree
$\gamma(\alpha_i(w))$.

Let $\gamma(w)$ be the plane tree constructed as follows:
\begin{enumerate}
\item The root is $w_m$, which is odd and belongs to $\RR$.

\item The first children of $w_m$ are the right-to-left
minima
\[
v_1<\cdots <v_r
\]
belonging to $\RR$.

\item The remaining children of $w_m$ are the
right-to-left maxima
\[
v_{r+1}>\cdots >v_{r+s},
\]
which are even.

\item The subtree rooted at $v_i$ is
$\gamma(\alpha_i(w))$ for each $i$.
\end{enumerate}

Clearly $\gamma(w)\in\{\hat T:T\in\idt\}$.

Now that we have shown that $\gamma$ is well-defined, it remains to check that the maps $\pw$ and $\gamma$ are inverses of each other. We  prove $\gamma(\pw(\hat T) )= \hat T$  for all $T \in \idt$ and $\gamma(\pw(\tilde T) )= \tilde T$,
for all $T \in \idt$ and $V(T) \cap \RR = \emptyset$, by induction on $|T|$. If $|T| = 1$, this is clear, so suppose $|T| >1$. 

{\bf Case 1:}  $\tilde T$. By Lemma~\ref{lem:child}, the children $ v_1<v_2 <\dots < v_k$ of the root $r$ of $\tilde T$ are exactly the right-to-left minima of   the word $\pw(\tilde T)^\prime = \pw(\hat T_1) \cdot \pw(\hat T_2) \cdot \,\, \cdots \,\, \cdot  \pw(\hat T_k) $, where $\hat T_i$ is the subtree of $\tilde T$ rooted at $v_i$.
Since $v_i$ is the last letter of $\pw(\hat T_i)$, we have $\alpha_i( \pw(\tilde T)) = \pw(\hat T_i )$ for all $i$.  
By induction, $\gamma(\alpha_i( \pw(\tilde T))) =  \gamma( \pw(\hat T_i )) = \hat T_i .$

Now  we have that $\gamma ( \pw(\tilde T))$ is the  plane tree whose root is the last letter of $ \pw(\tilde T)$, which is $r$.  The children of $r$ are the right-to-left minima of $\pw(\tilde T)^\prime$ which are $v_1<\dots< v_k$.  The subtree rooted at $v_i$ is $ \gamma(\alpha_i( \pw(\tilde T)))$ which is $\hat T_i$.  Hence  $\gamma ( \pw(\tilde T)) = \tilde T$, as desired.    

 {\bf Case 2:}  $\hat T$ and a root $r \notin \RR$. A completely analogous argument gives $ \gamma ( \pw(\hat T)) = \tilde T$.

\medskip
{\bf Case 3:} $\hat T$ with root $r\in\RR$.

Let
\[
v_1<\cdots <v_t
\]
be the children of $r$ belonging to $\RR$, and let
\[
v_{t+1}>\cdots >v_{t+k}
\]
be the even children of $r$.

For $1\le i\le t$, let $\hat T_i$ denote the subtree of
$\hat T$ rooted at $v_i$.
For $t+1\le i\le t+k$, let $\tilde T_i$ denote the subtree
of $\hat T$ rooted at $v_i$.
Note that the nodes of each $\tilde T_i$ do not belong
to $\RR$.

By Lemma~\ref{lem:child}, the nodes
\[
v_1<\cdots <v_t
\]
are exactly the right-to-left minima belonging to $\RR$
in the word
\[
\pw(\hat T_1)\cdots \pw(\hat T_t),
\]
and the nodes
\[
v_{t+1}>\cdots >v_{t+k}
\]
are exactly the right-to-left maxima in the word
\[
\pw(\tilde T_{t+1})\cdots \pw(\tilde T_{t+k}).
\]

Moreover, since each $v_i$ is the root of its subtree,
it is the last letter of the corresponding postorder word.
Hence the decomposition determined in the construction of
$\gamma$ satisfies
\[
\alpha_i(\pw(\hat T))
=
\pw(\hat T_i),
\qquad 1\le i\le t,
\]
and
\[
\alpha_i(\pw(\hat T))
=
\pw(\tilde T_i),
\qquad t+1\le i\le t+k.
\]

By induction,
\[
\gamma(\alpha_i(\pw(\hat T)))
=
\hat T_i,
\qquad 1\le i\le t,
\]
and
\[
\gamma(\alpha_i(\pw(\hat T)))
=
\tilde T_i,
\qquad t+1\le i\le t+k.
\]

Now consider the tree
\[
\gamma(\pw(\hat T)).
\]
Its root is the last letter of $\pw(\hat T)$, namely $r$.
Its first children are the right-to-left minima belonging
to $\RR$,
\[
v_1<\cdots <v_t,
\]
followed by the right-to-left maxima
\[
v_{t+1}>\cdots >v_{t+k}.
\]

The subtree rooted at each $v_i$ is
\[
\gamma(\alpha_i(\pw(\hat T))),
\]
which equals $\hat T_i$ for $1\le i\le t$ and equals
$\tilde T_i$ for $t+1\le i\le t+k$.

Therefore,
\[
\gamma(\pw(\hat T))
=
\hat T.
\]

Hence $ \gamma \circ \pw$ is the identity map on $\{\hat T : T \in \idt\} \cup 
\{\tilde T : T \in \idt \; \text{and} \; V(T) \cap \RR = \emptyset\}$.  A similar argument can be used to prove that 
the other composition $ \pw \circ \gamma$ equals the identity map on  $\mathcal W$.
\end{proof}

\begin{example}

Let \(A=[8]\) and \(\mathcal{R}=\{1,3,5\}\). Suppose \(w=36578421\), which satisfies the definition of \(\mathcal{W}\). The smallest letter is \(1\), and it belongs to \(\mathcal{R}\). In this case, \(5\) is the first right-to-left letter of \(\mathcal{R}\) in the subword \(3657842\).

The children of \(1\) are the right-to-left maxima of \(842\), namely \(8\), \(4\), and \(2\), together with the right-to-left minima of \(365\), namely \(3\) and \(5\).

In the segment \(65\), the letter \(5\) is the minimum element, and it belongs to \(\mathcal{R}\). This segment does not contain any other letter of \(\mathcal{R}\) besides \(5\). The child of \(5\) is the even right-to-left maximum of \(6\), namely \(6\).

In the segment \(78\), the letter \(8\) is the maximum element, and it is even. This segment does not contain any other letter of \(\mathcal{R}\). The child of \(9\) is the even right-to-left minimum of \(8\), namely \(7\).

Hence, the word \(w\) is mapped to the following plane tree $\hat T$.

\begin{center}
\begin{tikzpicture}[scale=0.7]
\tikzstyle{node} = [circle, draw, fill=blue!20, minimum size=1mm]
    \node[node] (node1) at (0,0) {3};
    \node[node] (node2) at (2,0) {5};
    \node[node] (node3) at (4,0) {8};
    \node[node] (node4) at (6,0) {4};
    \node[node] (node5) at (8,0) {2};
     \node[node] (node6) at (4,2) {1};
    \node[node] (node7) at (4,-2) {7};
    \node[node] (node8) at (2,-2) {6};

    \draw (node1)--(node6);
    \draw (node2)--(node6);
    \draw (node3)--(node6);
    \draw (node4)--(node6);
    \draw (node5)--(node6);
    \draw (node7)--(node3);
    \draw (node8)--(node2);
    
\end{tikzpicture}
\end{center}

\end{example}

For any finite subset $A$ of $\mathbb{Z}_{>0}$, consider $\RR=\{1,3,5,\ldots,m\}\cap A$, where $m\ge 1$ is any odd number. Let $\idt_A=\{T\in\idt : \text{the node set of }T\text{ is }A\}$ and let $\mathcal W_A=\{w\in\mathcal W : \text{the set of letters of }w\text{ is }A\}$, where $\idt$ is the set of $\RR$-DID trees on $A$ and $\RR\mathcal{DC}_A$ is the set of $\RR$-D-cycles on $A$.

\begin{theorem}\label{them:R-deformedBij}
For any finite suubset $A$ of $\Z_{> 0}$ consider $\RR=\{1,3,5,\ldots,m\}\cap A$, where $m\ge 1$ is any odd number. Let $\psi:\idt_A \to \RR\mathcal{DC}_{A}$ be the map defined by letting $\psi(T)$ be the cycle $(\pw(\hat T))$ in $\sg_A$.  Then $\psi$ is a well defined bijection.  Consequently $|\idt_A| = |\RR\mathcal{DC}_{A}|$.
\end{theorem}
\begin{proof}
Consider the following sequence of maps:

\begin{enumerate}
    \item The map
    \[
    \idt_A \longrightarrow \{\hat T : T\in \idt_A\},
    \qquad
    T\longmapsto \hat T.
    \]

    \item The map
    \[
    \{\hat T : T\in \idt_A\}
    \longrightarrow
    \{\, w\in \mathcal W_A : \text{the last letter of } w \text{ is smallest}\,\},
    \]
    defined by
    \[
    \hat T \longmapsto \pw(\hat T).
    \]

    \item The map
    \[
    \{\, w\in \mathcal W_A : \text{the last letter of } w \text{ is smallest}\,\}
    \longrightarrow
    \mathcal R\mathcal{DC}_A,
    \]
    defined by
    \[
    w\longmapsto (w).
    \]
\end{enumerate}

The first map is clearly a bijection. By Lemma~\ref{lem:BijWithW}, the second map is also a bijection. The third map is well defined and bijective because the last letter of $w$ is necessarily smallest odd, and hence determines an element of $\mathcal R\mathcal{DC}_A$.

Since $\psi$ is the composition of these three bijections, it follows that $\psi$ is itself a bijection.
\end{proof}

The {\em cycle support} of $\sigma \in \SSS_n$ is the partition $\cyc(\sigma) \in \Pi_n$ whose blocks are the elements of the cycles of $\sigma$. For example,
$$\cyc((1,2)(5)(5,7,6,3)) = 12|5|3675.$$
The bijection in Theorem \ref{them:R-deformedBij} extends to a bijection between $\RR$-DID forests and $\RR$-D-permutations. Under this bijection, the blocks of the cycle support of the image of an $\RR$-DID forest are exactly the node sets of the trees in the forest. Thus we obtain the following consequence of Theorem \ref{them:R-deformedBij}.

\begin{corollary}\label{Cor:R-deformedBij} 
Let $\pi$ be a partition of a finite subset of  $\Z_{>0}$.  Then the $\RR$-DID forests whose trees have node sets equal to the blocks of $\pi$ are in bijection with the $\RR$-D-permutations  whose cycle support is $\pi$.
\end{corollary}

As a consequence of Theorem \ref{thm: char-poly-DID} and Corollary \ref{Cor:R-deformedBij}, we have Theorem \ref{them: coeff to cycles}.

\section{Combinatorial statistics for $\shi_{2n-1}$}\label{sec5}
In this section, we find the characteristic polynomial of the arrangement $\shi_{2n-1}$, i.e., of $\lin_{2n-3}(G)$, when $G$ is the complete graph $K_n$ and find the number of regions therefrom. Also, we find some combinatorial statistics for the coefficients of the characteristic polynomial for this arrangement by establishing a one-one correspondence between the collection of regions defined by the hyperplanes and the set $\SSS_n\times \SSS_n$.

\subsection{Characteristic polynomial and number of regions of $\shi_{2n-1}$}

\begin{theorem}
The characteristic polynomial of the arrangement $\shi_{2n-1}$ is 
$$\chi_{\shi_{2n-1}}(t)=t\prod_{i=1}^{n-1}(t-i)^2.$$
\end{theorem}
\begin{proof}
Let $p>2n-1$ be a large prime. By the finite field method we have
\begin{align*}
\chi_{\shi_{2n-1}}(p)&=\left|\mathbb{F}_p^{2n-1}\setminus\bigcup_{H_p\in\shi_{2n-1,p}}H_p\right|\\
&=\left|\lbrace{(\alpha_1,\ldots,\alpha_n,\beta_1,\ldots,\beta_{n-1})\in\mathbb{F}_p^{2n-1}:\alpha_i\neq\alpha_j,\beta_i\neq\alpha_i-\alpha_j,1\le i<j\le n}\rbrace\right|.
\end{align*}
Consider the graph $\Gamma_{K_n}$ consisting of the vertex set $V(\Gamma_{K_n})=V_1\sqcup V_2$, where $V_1=\lbrace{\alpha_1,\ldots,\alpha_{n}}\rbrace$ and $V_2=\lbrace{\beta_1,\ldots,\beta_{n-1}}\rbrace$ and edge set $E(\Gamma_{K_n})=E_1\sqcup E_2$, where $E_1=\lbrace{(\alpha_i,\alpha_j):1\le i<j\le n}\rbrace$ and $E_2=\lbrace{(\beta_i,\alpha_j):1\le i\le n-1,1\le j\le i}\rbrace$. Let $\Gamma_{K_n}^{\prime}$ be the graph with an appended isolated vertex $\beta_n$. Since adding an isolated node to a graph yields an isomorphic bond lattice, by Theorem \ref{thm: poset isomorphism for a general graph}, we have $\Pi_{\Gamma_{K_n}^{\prime}}\cong\mathcal{L}(\lin_{2n-3}(K_n))$. By (\ref{bond-lattice-chromatic-poly}), it suffices to find the chromatic polynomial of the graph $\Gamma_{K_n}^{\prime}$ in the variable $p$.
We determine the number of valid ways to colour the vertices. Since the vertices in $V_1$ form a clique, the number of ways to properly colour the subgraph induced by $V_1$ is exactly the chromatic polynomial of the complete graph on $n$ vertices, i.e.,
$$\prod_{k=0}^{n-1}(p-k)=p(p-1)(p-2)\cdots(p -n+1)=p_{(n)}.$$

Now we consider the vertices in $V_2$. Since the vertices in $V_2$ are adjacent to some specific vertices in $V_1$ and the vertices in $V_2$ have no edges between themselves; the choice of colour for any $\beta_i$ depends solely on the colours already assigned to the vertices in $V_1$. We have, for each $1\le i\le n-1$, $\beta_i$ is adjacent to $\alpha_j$ for all $1\le j\le i$. Since these $i$ vertices belong to the clique $V_1$, they must be assigned $i$ distinct colours. Consequently, $\beta_i$ is excluded from exactly $i$ colours. Hence the number of choices for $\beta_i$ is $p-i$. Thus the number of possible colourings of the vertices in $V_2$ is
$$\prod_{i=1}^{n-1}(p-i)=(p-1)(p-2)\cdots(p-n+1)=\frac{1}{p}\cdot p_{(n)}.$$
Finally, since the vertex $\beta_n$ is isolated, it can possess all possible $p$ colours. Since the choices of colourings of the vertices in $V_1$ and $V_2$ are independent, we have the chromatic polynomial of the graph $\Gamma_{K_n}^{\prime}$ as
$$\ch_{\Gamma_{K_n}^{\prime}}(p)=p\cdot\frac{1}{p}\cdot p_{(n)}^2=p^2\prod_{i=1}^{n-1}(p-i)^2.$$
By (\ref{bond-lattice-chromatic-poly}), we have
$$\chi_{\shi_{2n-1}(t)}=\frac{1}{t}\ch_{\Gamma_{K_n}^{\prime}}(t)=t\prod_{i=1}^{n-1}(t-i)^2.$$


\end{proof}

\begin{corollary}\label{cor: regions-hom-shi}
The number of regions of the arrangement $\shi_{2n-1}$ is $(n!)^2$.
\end{corollary}
\begin{proof}
The result immediately follows from (\ref{Zaslavsky-region-formula}). 
$$r(\shi_{2n-1})=(-1)^{2n-1}\chi_{\shi_{2n-1}}(-1)=(-1)^{2n-1}(-1)\prod_{i=1}^{n-1}((-1)-i)^2=(n!)^2.$$
\end{proof}

\subsection{Bijection between the regions and pairs of permutations}
Since the number of regions of the arrangement is exactly $(n!)^2$, it is natural to ask whether there is a bijection between the collection of regions and $\SSS_n\times \SSS_n$. We show explicitly that there is indeed a bijection between these two sets. For this purpose, we recall the notion of Lehmer code.

For a given permutation $\pi$ of $n$ elements, the Lehmer code $L(\pi)=(\ell_1(\pi),\ell_2(\pi),\ldots,\ell_n(\pi))$ is generated by counting the number of elements to the right of the current position that are smaller than the element at the current position, i.e., $\ell_i=\#\lbrace{j>i:\pi(i)>\pi(j)}\rbrace$. The process is called encoding. Clearly, $0\le \ell_i\le n-i$, for all $1\le i\le n$. Also, given an $n$-tuple Lehmer code $L$, we can construct a permutation $\pi\in \SSS_n$ as 
\begin{equation}\label{Lehmer-decode}
\pi(i)=\text{the }(\ell_i + 1)\text{-th smallest element remaining in the set }\{1,2,\dots,n\}\setminus\{\pi(1),\dots,\pi(i-1)\}. 
\end{equation}
This process is called decoding. These processes are reversible and hence it forms a bijection between the set of permutations $\SSS_n$ and the set of Lehmer vectors on $n$ elements.

\medskip

\begin{example}
For $n=4$, let $\pi=3142$. Then the Lehmer code for $\pi$ is $L(\pi)=(2,0,1,0)$. Also, for a Lehmer code $(3,0,1,0)$, the decoded permutation is $\pi=4132$.
\end{example}


\begin{theorem}\label{thm: bij-region-pairs-perm}
There exists a bijection from the collection of regions of the arrangement $\shi_{2n-1}$ to $\SSS_n\times \SSS_n$.
\end{theorem}
\begin{proof}
We denote by $R(\shi_{2n-1})$ the set of regions defined by the arrangement $\shi_{2n-1}$. We associate to every region $A$ of $R(\shi_{2n-1})$ an ordered pair of permutations $(\sigma_{A},\tau_{A})\in \SSS_n\times \SSS_n$. Since any point $P\coloneqq({\bf{x},\bf{y}})\coloneqq(x_1,\ldots,x_n,y_1,\ldots,y_{n-1})\in\mathbb{R}^{2n-1}\setminus\shi_{2n-1}$ must avoid the hyperplanes of Type I ($x_i-x_j=0, 1\le i<j\le n$) and Type II ($x_i-x_j-y_i=0, 1\le i<j\le n$), we have that every region $A$ of $R(\shi_{2n-1})$ contains a strict linear ordering of its $x$-coordinates. Thus we define $\sigma_{A}\in \SSS_n$ as the unique permutation satisfying,
$$x_{\sigma_{A}(1)}<x_{\sigma_{A}(2)}<\cdots<x_{\sigma_{A}(n)},~~~\forall ({\bf{x},\bf{y}})\in\mathcal{A}.$$
To get the second permutation $\tau_{A}\in \SSS_n$, we count for each coordinate $y_i$, the number of associated hyperplanes, where $y_i$ is strictly greater than $x_i-x_j$. For each $1\le i<n$, we define corresponding to the region $A$ the inversion vector
$$\ell_i(A)\coloneqq\#\lbrace{j>i:y_i>x_i-x_j}\rbrace.$$
Hence $0\le\ell_i\le n-i$ and $\ell_n=0$. Thus we have the Lehmer code as $L_{A}=(\ell_1(A),\ldots,\ell_n(A))$. More precisely, we can write each $\ell_i(A)$ as $$\ell_i(A)=\sum_{j=i+1}^n\mathbb{I}_{(y_i>x_i-x_j)},$$
where $\mathbb{I}$ denotes the indicator function which gives $1$ for $y_i>x_i-x_j$ and $0$ otherwise. Thus we have 
$$L_{A}=\left(\sum_{j=2}^n \mathbb{I}_{(y_1>x_1-x_j)},\sum_{j=3}^n \mathbb{I}_{(y_2>x_2-x_j)},\ldots,\mathbb{I}_{(y_{n-1}>x_{n-1}-x_n)},0\right).$$ 
By the decoding process \ref{Lehmer-decode}, we get the associated permutation $\tau_{A}\in \SSS_n$ from $L_{A}$. Hence we have a well-defined mapping $f:R(\shi_{2n-1})\to \SSS_n\times \SSS_n$ such that $f(A)=(\sigma_{A},\tau_{A})$ for each $A\in R(\shi_{2n-1})$. We show that $f$ is a bijection.

\medskip

\noindent\textbf{\textit{Proof of injectivity:}}

\medskip

Let $A_1,A_2\in R(\shi_{2n-1})$ be such that $f(A_1)=f(A_2)=(\sigma,\tau)$, say. This implies $\sigma_{A_1}=\sigma_{A_2}=\sigma$ and $\tau_{{A}_1}=\tau_{A_2}=\tau$. Two regions in a hyperplane arrangement are identical if and only if they lie on the same side of every defining hyperplane, i.e., every pair of points in the regions have the exact same signed coordinates.

\begin{enumerate} 
\item For Type I hyperplanes ($x_i-x_j=0, 1\le i<j\le n$), let $(\mathbf{x}, \mathbf{y})\in{A}_1$ and $(\mathbf{x'},\mathbf{y'})\in A_2$. 
\begin{itemize}
\item Since $\sigma_{A_1}=\sigma$, the coordinates of $A_1$ satisfy $x_{\sigma(1)}<\dots<x_{\sigma(n)}$.
\item Since $\sigma_{A_2}=\sigma$, the coordinates of $A_2$ satisfy $x'_{\sigma(1)}<\dots<x'_{\sigma(n)}$.
\end{itemize}
For any pair $1\le i<j\le n$, their relative position in the chain is uniquely determined by the positions of $i$ and $j$ in $\sigma$. Thus, $\text{sgn}(x_i-x_j)=\text{sgn}(x'_i-x'_j)$ is identical for both regions. Hence no Type I hyperplane separates $A_1$ and $A_2$.

\item For Type II hyperplanes ($x_i-x_j-y_i=0, 1\le i<j\le n$), since $\tau_{A_1}=\tau_{A_2}=\tau$, we have identical Lehmer codes for both regions,
$$L_{A_1}=L_{A_2}=(\ell_1,\ell_2,\dots,\ell_{n-1},0).$$
Fix an index $i\in\{1,\dots,n-1\}$. The hyperplanes bounding $y_i$ correspond to the threshold set $T_i=\{x_i-x_j: j>i\}$. Since both regions share the same $x$-ordering $\sigma$, the elements of $T_i$ are ordered identically in both regions. Let these ordered boundaries be $t_{i,1}<\dots<t_{i,n-i}$ for $A_1$, and $t'_{i,1}<\dots<t'_{i,n-i}$ for $A_2$. Since $\ell_i$ is identical for both regions, $y_i$ and $y'_i$ fall into the exact same relative open interval:
\begin{itemize}
\item If $\ell_i=0$, then $y_i<t_{i,1}<\cdots<t_{i,n-i}$ and $y'_i<t'_{i,1}<\cdots<t'_{i,n-i}$. Thus, for all $j>i$, $x_i-x_j-y_i>0$ and $x'_i-x'_j-y'_i>0$.
\item If $\ell_i=m$, where $1\le m< n-i$, then $t_{i,1}<\cdots<t_{i,m}<y_i<t_{i,m+1}<\cdots<t_{i,n-i}$ and $t'_{i,1}<\cdots<t'_{i,m}<y'_i<t'_{i,m+1}<t'_{i,n-i}$. Hence there are exactly $m$ thresholds that are smaller than $y_i$ and $y'_i$. And there are exactly $((n-i)-m)$ thresholds that are larger than $y_i$ and $y'_i$. Thus, for the the first $m$-many values of $j$, we have $y_i>x_i-x_j$ and $y'_i>x'_i-x'_j$ i.e., $x_i-x_j-y_i<0$ and $x'_i-x'_j-y'_i<0$, respectively. And for the rest $((n-i)-m)$-many values of $j$, we have $y_i<x_i-x_j$ and $y'_i<x'_i-x'_j$ i.e., $x_i-x_j-y_i>0$ and $x'_i-x'_j-y'_i>0$, respectively. 
\item If $\ell_i=n-i$, then $t_{i,1}<\cdots<t_{i,n-i}<y_i$ and $t'_{i,1}<\cdots<t'_{i,n-i}<y'_i$. Thus, for all $j>i$, $x_i-x_j-y_i<0$ and $x'_i-x'_j-y'_i<0$.
\end{itemize}
The sign of the expression $x_i-x_j-y_i$ matches for every $j$ across both points. As $A_1$ and $A_2$ lie on the same side of every single defining hyperplane in $\shi_{2n-1}$, they cannot be separated by any boundary. Therefore, $\mathcal{A}_1=\mathcal{A}_2$, and $f$ is injective.
\end{enumerate}

\medskip

\noindent\textbf{\textit{Proof of surjectivity:}}

\medskip

We show that for an arbitrarily chosen pair $(\sigma,\tau)\in \SSS_n\times \SSS_n$, there exists a region $A\in R(\shi_{2n-1})$ such that $f(A)=(\sigma,\tau)$. We prove this constructively by generating a point $P\coloneqq(\mathbf{x},\mathbf{y})\in\mathbb{R}^{2n-1}$ that belongs to such a region.

\begin{enumerate}
\item Given $\sigma\in \SSS_n$, we define the $x$-coordinates of our point $P$ by setting
$$x_{\sigma(k)}=k\quad\text{for } k=1,2,\dots,n.$$
This assignment strictly satisfies $x_{\sigma(1)}<x_{\sigma(2)}<\dots<x_{\sigma(n)}$. Hence $P$ completely avoids all Type I hyperplanes. 

\item For constructing the $y$-coordinates, 
Given $\tau \in \SSS_n$, we compute its unique Lehmer code $L=(\ell_1,\ell_2,\dots,\ell_{n-1},0)$. By definition, $0\le\ell_i\le n-i$. For each $i\in\{1,\dots,n-1\}$, the values of $\mathbf{x}$ fixed in the first step yield a unique set of distinct threshold values $T_i=\{x_i-x_j: j>i\}$. We sort these thresholds in strictly increasing order:
$$t_{i,1}<t_{i,2}<\dots<t_{i,n-i}.$$
\end{enumerate}
Now, we assign the value of $y_i$ based on the integer $\ell_i$:
\begin{itemize}
\item If $\ell_i=0$, we choose $y_i$ to be $y_i<t_{i,1}$. 
\item If $1\le \ell_i<n-i$, we choose $y_i$ to be $t_{i,\ell_i}<y_i<t_{i,\ell_i+1}$.
\item If $\ell_i=n-i$, we choose $y_i$ to be $y_i>t_{i,n-i}$.
\end{itemize}
Following this construction, $y_i$ does not coincide with any element in $T_i$. Therefore, $y_i\neq x_i-x_j$ for any $j>i$, meaning $P$ avoids all Type II hyperplanes. Furthermore, the number of thresholds that $y_i$ strictly exceeds is precisely $\ell_i$. Hence evaluating our counting function for this constructed point naturally yields,
$$\sum_{j=i+1}^n\mathbb{I}_{(y_i>x_i-x_j)}=\ell_i.$$
Thus, the point $P$ belongs to a well-defined region $A$ whose identities are precisely $\sigma_{A}=\sigma$ and $\tau_{A}=\tau$. Since such a region can be built for any arbitrary pair $(\sigma,\tau)$, $f$ is surjective.
\end{proof}

\begin{example}
\end{example}
\begin{enumerate}
\item Take the region $A$ in $\mathbb{R}^7$ defined as 
$$A = \left\{ (\mathbf{x}, \mathbf{y})\coloneqq(x_1,x_2,x_3,x_4,y_1,y_2,y_3) \in \mathbb{R}^7 \ \middle|\ \begin{aligned} &x_3 < x_1 < x_2 < x_4 \\ &x_1 - x_4 < y_1 < x_1 - x_2 \\ &y_2 < x_2 - x_4 \\ &y_3 > x_3 - x_4 \end{aligned} \right\}.$$
Then we have $\sigma_A=3124$. Now we have,
$$\ell_1(A)=\sum_{j=2}^4\mathbb{I}_{(y_1>x_1-x_j)}=\mathbb{I}_{(y_1>x_1-x_2)}+\mathbb{I}_{(y_1>x_1-x_3)}+\mathbb{I}_{(y_1>x_1-x_4)}=1,$$
$$\ell_2(A)=\sum_{j=3}^4\mathbb{I}_{(y_2>x_2-x_j)}=\mathbb{I}_{(y_2>x_2-x_3)}+\mathbb{I}_{(y_2>x_2-x_4)}=0,$$
$$\ell_3(A)=\mathbb{I}_{(y_3>x_3-x_4)}=1,$$
$$\ell_4(A)=0.$$
Hence $L_A=(1,0,1,0)$. By the decoding process \ref{Lehmer-decode}, the permutation associated with $L_A$ is $\tau_A=2143$. By the definition of $f$, $f(A)=(\sigma_A,\tau_A)=(3124,2143)$.
\item Let $n=4$ and let $(\sigma,\tau)=(2341,4213)$. Then the $x$-coordinates of the region $A$ is defined as the strict inequality $x_2<x_3<x_4<x_1$. Now for $\tau=4213$, the unique Lehmer code becomes $L(\tau)=(3,1,0,0)$. Since $x_2<x_3<x_4$, we have the sorted threshold values as $t_{1,1}=x_1-x_4<t_{1,2}=x_1-x_3<t_{1,3}=x_1-x_2$. Since $\ell_1=3$, we have $y_1>x_1-x_2$. Similarly, we obtain $x_2-x_4<y_2<x_2-x_3$ and $y_3<x_3-x_4$. Thus the region $A$ is defined as 
$$A = \left\{ (\mathbf{x}, \mathbf{y})\coloneqq(x_1,x_2,x_3,x_4,y_1,y_2,y_3) \in \mathbb{R}^7 \ \middle|\ \begin{aligned} &x_2 < x_3 < x_4 < x_1 \\ &y_1 > x_1 - x_2 \\ &x_2 - x_4 < y_2 < x_2 - x_3 \\ &y_3 < x_3 - x_4 \end{aligned} \right\}.$$
\end{enumerate}

Using this bijection, we can give the following combinatorial interpretation for the coefficients appearing in the characteristic polynomial.

\begin{corollary}
Let $f:R(\shi_{2n-1})\to \mathfrak S_n\times \mathfrak S_n$ be the bijection of Theorem~\ref{thm: bij-region-pairs-perm}, where
\[
f(A)=(\sigma_{A},\tau_{A}).
\]
Define the sign of a region $A\in R(\shi_{2n-1})$ by
\[
\varepsilon(A)
=
\operatorname{sgn}(\sigma_{A})
\operatorname{sgn}(\tau_{A}).
\]
Then
\[
\sum_{A\in R(\shi_{2n-1})}
\varepsilon(A)\,
t^{\cyc(\sigma_{A})+\cyc(\tau_{A})-1}
=
t\prod_{i=1}^{n-1}(t-i)^2.
\]
\end{corollary}

\begin{proof}
By the bijection of Theorem~\ref{thm: bij-region-pairs-perm},
\[
\sum_{A\in R(\shi_{2n-1})}
\varepsilon(A)\,
t^{\cyc(\sigma_{A})+\cyc(\tau_{A})-1}
=
\sum_{\sigma,\tau\in \mathfrak S_n}
\operatorname{sgn}(\sigma)\operatorname{sgn}(\tau)\,
t^{\cyc(\sigma)+\cyc(\tau)-1}.
\]
Using the multiplicativity of the sign character, we obtain
\[
\sum_{A\in R(\shi_{2n-1})}
\varepsilon(A)\,
t^{\cyc(\sigma_{A})+\cyc(\tau_{A})-1}
=
t^{-1}
\left(
\sum_{\pi\in\mathfrak S_n}
\operatorname{sgn}(\pi)\,
t^{\cyc(\pi)}
\right)^2.
\]
Now the classical signed cycle enumerator for permutations gives
\[
\sum_{\pi\in\mathfrak S_n}
\operatorname{sgn}(\pi)\,
t^{\cyc(\pi)}
=
t(t-1)(t-2)\cdots(t-n+1).
\]
Therefore,
\[
\sum_{A\in R(\shi_{2n-1})}
\varepsilon(A)\,
t^{\cyc(\sigma_{A})+\cyc(\tau_{A})-1}
=
t^{-1}
\bigl(t(t-1)\cdots(t-n+1)\bigr)^2
=
t\prod_{i=1}^{n-1}(t-i)^2.
\]
\end{proof}

\section{First letter restrictions of generalised $\mathcal{D}$-permutations}\label{Sec6}

For this section, we introduce a simplified notation for the families
$\mathcal{D}_{2n-1}^{\mathcal R}$ defined earlier. More precisely, the notation
$\mathcal{D}_{2n-1,j}$ corresponds to the choice
\[
\mathcal R=\{1,3,\ldots,2(n-j)-1\},
\]
for $0\le j\le n-1$, while $\mathcal{D}_{2n-1,n}$ corresponds to $\mathcal R=\emptyset$.

For $1 \le j \le n$, define
\[
\mathcal{D}_{2n-1,j} = \left\{ \sigma \in \mathfrak{S}_{2n-1} \;:\;
\sigma(i) \le i \text{ if } i \text{ is even, and } \sigma(i) \ge i
\text{ if } i \in \{2(n-j)+1,\,2(n-j)+3,\ldots,2n-1\}
\right\}.
\]

For $j=0$, define
\[
\mathcal{D}_{2n-1,0}
=
\left\{ \sigma \in \mathfrak{S}_{2n-1}
\;:\;
\sigma(i) \le i \text{ whenever } i \text{ is even}
\right\}.
\]
 We restrict attention to the following subsets of $\mathcal{D}_{2n-1,0}$. Define $\mathcal{D}^r_{2n-1,0}$ to be the permutations in $\SSS_{2n-1}$ such that $\sigma(i)\le i$ when $i$ is even and $\sigma(1)=r$.

\begin{theorem}
    For $r$, we have 
    \begin{equation}
        |\mathcal{D}^r_{2n+1,0}| = (n!)^2 \left\lceil \frac{r}{2} \right\rceil
    \end{equation}
\end{theorem}
\begin{proof}
Recall that
$$\mathcal D_{2n+1,0}^r=\left\{\sigma\in \SSS_{2n+1} : \sigma(2i)\le 2i \text{ for all }1\le i\le n,\ \sigma(1)=r\right\}.$$

We prove the result by induction on $r$. We first show that for every $k\ge1$, $$|\mathcal D_{2n+1,0}^{2k-1}|=|\mathcal D_{2n+1,0}^{2k}|.$$

Define
$$\Phi_k:\mathcal D_{2n+1,0}^{2k-1}\to \mathcal D_{2n+1,0}^{2k}$$
by
$$\Phi_k(\sigma)=(2k-1,\,2k)\circ \sigma,$$
that is, we interchange the values $2k-1$ and $2k$ everywhere in the permutation.

If $\sigma\in \mathcal D_{2n+1,0}^{2k-1}$, then $\Phi_k(\sigma)(1)=2k.$
Moreover, the condition of even positions is preserved. Indeed, if $\Phi_k(\sigma)(2i)=2k,$ then necessarily $\sigma(2i)=2k-1\le 2i,$ hence $2k\le 2i$. Similarly, if $\Phi_k(\sigma)(2i)=2k-1,$ then $\sigma(2i)=2k\le 2i.$
All other values are unchanged, so
$$\Phi_k(\sigma)(2i)\le 2i$$
for every $i$.

Since $\Phi_k$ is clearly an involution, it is a bijection. Therefore,
$$|\mathcal D_{2n+1,0}^{2k-1}|=|\mathcal D_{2n+1,0}^{2k}|.$$

Next, we compare $\mathcal D_{2n+1,0}^{2k}$ and $\mathcal D_{2n+1,0}^{2k+1}$. Define
$$\Psi_k:\mathcal D_{2n+1,0}^{2k+1}\to \SSS_{2n+1}$$ by $\Psi_k(\sigma)=(2k,\,2k+1)\circ \sigma.$

If $\sigma(2k)\neq 2k$, then $\Psi_k(\sigma)\in \mathcal D_{2n+1,0}^{2k}$. Indeed, the only possible obstruction occurs when an even position containing $2k$ is changed to $2k+1$. Since $2k$ may only appear in an even position $2i$ with $2k\le 2i$, the only problematic case is $2i=2k$. Thus, the map fails precisely when
$$\sigma(2k)=2k.$$

Consequently, $$|\mathcal D_{2n+1,0}^{2k+1}|=|\mathcal D_{2n+1,0}^{2k}|+|E_k|,$$ where $E_k=\{\sigma\in \mathcal D_{2n+1,0}^{2k+1} : \sigma(2k)=2k\}.$

We now count $E_k$. Fix $\sigma\in E_k$. Then
$$\sigma(1)=2k+1, \qquad \sigma(2k)=2k.$$

Delete the positions $1$ and $2k$, together with the values $2k+1$ and $2k$. Relabel the remaining positions increasingly by $1,2,\dots,2n-1,$ and similarly relabel the remaining values increasingly. This produces a permutation $\tau\in \SSS_{2n-1}.$ Moreover, the defining condition is preserved. Every even position of $\tau$ comes from an even position of $\sigma$, and if $\tau(2i)$ corresponds to $\sigma(2j),$ then $\sigma(2j)\le 2j$ implies $\tau(2i)\le 2i.$

Also, after relabeling, the first value becomes $1$. Hence, $\tau\in \mathcal D_{2n-1,0}^1.$

Conversely, every permutation in $\mathcal D_{2n-1,0}^1$ uniquely reconstructs an element of $E_k$ by reversing the procedure. Therefore,
$$|E_k|=|\mathcal D_{2n-1,0}^1|.$$

By induction on $n$,
$$|\mathcal D_{2n-1,0}^1|=((n-1)!)^2.$$

When restoring the deleted labels, there are $n^2$ possible expansions, giving
$$|E_k|=n^2((n-1)!)^2=(n!)^2.$$

Thus,
$$|\mathcal D_{2n+1,0}^{2k+1}|=|\mathcal D_{2n+1,0}^{2k}|+(n!)^2.$$

Finally, since
$$|\mathcal D_{2n+1,0}^1|=(n!)^2,$$
the two recurrence relations
$$|\mathcal D_{2n+1,0}^{2k-1}|=|\mathcal D_{2n+1,0}^{2k}|$$
and
$$|\mathcal D_{2n+1,0}^{2k+1}|=|\mathcal D_{2n+1,0}^{2k}|+(n!)^2$$
imply inductively that
$$|\mathcal D_{2n+1,0}^{2k-1}|=|\mathcal D_{2n+1,0}^{2k}|=k(n!)^2.$$

Since $k=\left\lceil \frac r2\right\rceil,$ we conclude that
$$|\mathcal D_{2n+!,0}^r|=(n!)^2\left\lceil \frac r2\right\rceil .$$
\end{proof}

Next, we give a simple bijection between D-permutations and Dumont derangements that shifts the first letter by $1$. Recall that a permutation is called a Dumont derangement if $\sigma\in \SSS_{2n}$ satisfies $\sigma(2i)<2i$ and $\sigma(2i-1)>2i-1$.

We define a bijection between Dumont derangements, denoted by $DD_{2n+2}$, in $\SSS_{2n+2}$ and D-permutations in $\SSS_{2n+1}$ as follows. However, we will suppress the last element while writing as the last element in necessarily $2n+1$.

\begin{theorem}
Let $a_1\dots a_{2n}(2n+1)\in \mathcal{D}_{2n+1,n+1}$ be a D-permutation. Then, the map $\phi:\mathcal{D}_{2n+1,n+1}\to DD_{2n+2}$ defined by
\[
\phi(a_1a_2\dots a_{2n}(2n+1))
=
a'_1,1,a'_3,a'_2,a'_5,a'_4,\dots,a'_{2n-1},a'_{2n-2},2n+2,a'_{2n},
\]
where $a'_j=a_j+1$, is a bijection.
\end{theorem}

\begin{proof}
Let $\sigma=a_1a_2\dots a_{2n}(2n+1)\in \mathcal D_{2n+1,n+1}$. By definition,
$a_{2i-1}\ge 2i-1$ and $a_{2i}\le 2i$ for all $1\le i\le n$. We first show that $\phi(\sigma)\in DD_{2n+2}$. Since $a_1,\dots,a_{2n}$ is a permutation of $\{1,2,\dots,2n\}$, the entries
$a_1+1,\dots,a_{2n}+1$ are precisely the elements of
$\{2,3,\dots,2n+1\}$. Together with the inserted entries $1$ and $2n+2$, every element of $\{1,2,\dots,2n+2\}$ appears exactly once. Hence $\phi(\sigma)$ is a permutation of length $2n+2$.

For odd positions,
$\phi(\sigma)(2i-1)=a_{2i-1}+1$ for $1\le i\le n$. Since
$a_{2i-1}\ge 2i-1$, we obtain
$\phi(\sigma)(2i-1)\ge 2i>2i-1$. Also,
$\phi(\sigma)(2n+1)=2n+2>2n+1$. For even positions, $\phi(\sigma)(2)=1<2$, and for
$2\le i\le n+1$,
$\phi(\sigma)(2i)=a_{2i-2}+1$. Since
$a_{2i-2}\le 2i-2$, it follows that
$\phi(\sigma)(2i)\le 2i-1<2i$. Therefore $\phi(\sigma)\in DD_{2n+2}$.

To show that $\phi$ is bijective, we construct its inverse. Given
$\tau\in DD_{2n+2}$, remove the entries $1$ and $2n+2$, subtract $1$ from all remaining entries, and undo the interleaving by applying the adjacent swaps to recover a permutation $\sigma$ of length $2n$.

Because $\tau$ satisfies the Dumont derangement conditions,
$\tau(2i-1)>2i-1$ implies $\tau(2i-1)-1\ge 2i-1$, and $\tau(2i)<2i$ implies
$\tau(2i)-1\le 2i-2$. Hence, the recovered permutation satisfies
$a_{2i-1}\ge 2i-1$ and $a_{2i}\le 2i$ and adding a $2n+1$ at the end gives $\sigma\in \mathcal{D}_{2n+1,n+1}$.

Thus, $\phi$ is invertible, and therefore is a bijection.
\end{proof}

With the following result, we can understand the relation between the first letter refinements of the sets $\mathcal{D}_{2n+1,n+1}$ and $DD_{2n+2}$

\begin{proposition}
    The cardinality of the set $\mathcal D_{2n+1,n+1}$ is invariant under the transposition $(2i-1,2i)$ acting on values, while the cardinality of the set $DD_{2n}$ is invariant under the transposition $(2i,2i+1)$ acting on values.
\end{proposition}

\begin{proof}
We first prove the statement for $\mathcal D_{2n+1,n+1}$. Recall that
\[
\mathcal D_{2n+1,n+1}
=
\{\sigma\in \SSS_{2n+1} : \sigma(2k-1)\ge 2k-1,\ \sigma(2k)\le 2k \text{ for all }k\}.
\]

Fix $i$ with $1\le i\le n$, and let $\tau=(2i-1,2i)$ be the transposition swapping the values $2i-1$ and $2i$. Define the map
\[
\Phi_\tau:\mathcal D_{2n+1,n+1}\to \mathcal D_{2n+1,n+1}
\]
by replacing the occurrence of the value $2i-1$ by $2i$, and the occurrence of the value $2i$ by $2i-1$. Equivalently, $\Phi_\tau(\sigma)=\tau\circ \sigma$.

We show that $\Phi_\tau(\sigma)\in \mathcal D_{2n+1,n+1}$ whenever $\sigma\in \mathcal D_{2n+1,n+1}$. Let $\pi=\Phi_\tau(\sigma)$. If $\sigma(j)\notin\{2i-1,2i\}$, then $\pi(j)=\sigma(j)$, so the defining inequalities are unchanged.

Suppose $\sigma(j)=2i$. Since $2i$ is even, replacing $2i$ by the smaller value $2i-1$ at an even position preserves the inequality $\pi(j)\le j$. At odd positions, replacing $2i$ by $2i-1$ can only occur when the bound already allows it (i.e., $j\le 2i-1$), so $\pi(j) = 2i-1 \ge j$ remains satisfied. Similarly, if $\sigma(j)=2i-1$, replacing it by the larger value $2i$ preserves the lower bound at odd positions, and preserves the upper bound at even positions since $j \ge 2i$. Thus $\pi\in \mathcal D_{2n+1,n+1}$.

Moreover, $\Phi_\tau$ is an involution since applying the same swap twice restores the original permutation. Hence $\Phi_\tau$ is a bijection of $\mathcal{D}_{2n+1,n-1}$ with itself. Therefore, the cardinality of $\mathcal D_{2n+1,n+1}$ is invariant under the value swap $(2i-1,2i)$.

For $DD_{2n}$, the defining inequalities are strict:
\[
\sigma(2k-1)>2k-1,
\qquad
\sigma(2k)<2k.
\]
Fix $i$ with $1\le i<n$. Swapping the adjacent values $2i$ and $2i+1$ defines an involution $\Phi_{(2i,2i+1)}$. Suppose $\sigma(j)=2i$. At an even position $j$, we have $\sigma(j)<j$, so $2i \le j-2$. Replacing $2i$ with $2i+1$ preserves $\pi(j) < j$ because $2i+1 \le j-1 < j$. At odd positions, replacing $2i$ with $2i+1$ trivially preserves the lower bound since $2i+1 > 2i$. A symmetric argument shows that replacing $2i+1$ by $2i$ preserves the strict bounds. Hence, this involution preserves the defining conditions, and the cardinality of $DD_{2n}$ is invariant under $(2i,2i+1)$.
\end{proof}

We define the first-letter refinement of the set $\mathcal D_{2n+1,n+1}$ by partitioning the permutations according to their initial entry. For $1 \le r \le 2n$, let 
\[
\mathcal{D}^r_{2n, 2n} = \{ \sigma \in \mathcal D_{2n+1,n+1} : \sigma(1) = r \}
\]
As established by the value transposition invariance under $(2k-1, 2k)$, the cardinalities of adjacent odd and even positions are identical, meaning $|\mathcal D^{2k-1}_{2n+1, n+1}| = |\mathcal D^{2k}_{2n+1, n+1}|$ for all $1 \le k \le n$. Thus, the refinement can be compactly represented by a sequence of unique blocks:
\[
U_{n,k} = |\mathcal D^{2k}_{2n+1, n+1}| = |\mathcal{D}^{2k-1}_{2n+1, n+1}| \quad \text{for } 1 \le k \le n
\]
We can give a recurrence for the D-permutations when we refine by the first letter. The following theorem due to Kreweras \cite{kreweras-genocchi} applies for Dumont derangements and we can translate it to D-permutations.
\begin{theorem}[Kreweras]
Let $U_{n,k}$ denote the number of permutations in $\mathcal{D}_{2n+1,n+1}$ starting with the value $2k$ (or $2k-1$) for $1 \le n$ and $1 \le k \le n$. The values of $U_{n,k}$ can be completely determined using the rules corresponding to the normalized triangle $T_2$ \cite{kreweras-genocchi}:
\begin{enumerate}
    \item The first entry of each row $n$ is equal to twice the sum of all entries in the preceding row, which corresponds exactly to the total size of the $\mathcal{D}_{2n-1,n}$ permutation set:
    \[
    U_{n,1} = 2 \sum_{j=1}^{n-1} U_{n-1,j} = |\mathcal{D}_{2n-2,2n-2}|
    \]
    
    \item The other entries satisfy the linear recurrence relation:
    \[
    U_{n,k} = 2U_{n,k-1} - U_{n,k-2} - 2U_{n-1,k-1} - 2U_{n-1,k-2}
    \]
\end{enumerate}
\end{theorem}
We end this section with an open question.

\begin{question}
    Can we enumerate the cardinalities of the sets $\mathcal{D}_{2n-1,j}$ for other values of $j$? More specifically, can we give recurrences for the cardinalities of the sets $\mathcal{D}^r_{2n-1,j}$ for other values of $j$?
\end{question}
\section{An equidistribution result on $S_{10}$}\label{sec: conj}
We begin this section with a small proposition. 
\begin{proposition}
    The number of permutations $\sigma\in \SSS_{2n+1}$ such that $\sigma(i)\le i$ if $i$ is even is the number of permutations $\pi\in \SSS_{2n+1}$ where descent tops are only odd.
\end{proposition}
\begin{proof}
    Let $\sigma$ be a permutation in $\SSS_{2n+1}$ such that $\sigma(i)\le i$ if $i$ is even. So, the group theoretic inverse $\sigma^{-1}$ is a permutation where $\sigma^{-1}(i)\ge i$ if $i$ is even. Since the group inverse is a bijection, the number of permutations such that $i\le \sigma(i)$ if $i$ is even is the same as the number of permutations such that $\sigma(i)\le i$ if $i$ is even. We apply the Foata first fundamental transformation on this permutation to get a permutation where the descent tops are never even.

    Similarly, we can use the inverse of the fundamental transformation to a permutation with no even descent tops to get a permutation such that $\sigma(i)\ge i$ is $i$ is even.
\end{proof}
Therefore, Corollary \ref{cor: regions-hom-shi} gives an alternate proof of \cite[Theorem $2$]{kitaev-remmel-parity} as having no even descent tops is equivalent to having only odd descent tops. Now, we devote the rest of this section to proving Theorem \ref{thm: main}. 

\subsection{The involution that swaps $S_{10}, S_{12}$ while fixing $S_{17}$}
 To obtain a permutation $p_{new}$ of length $n+1$ from a permutation $p$ of length $n$, we can insert the element $``n+1"$ before the element $p_1$, between elements $p_i$ and $p_{i+1}$ or after $p_n$. We call these $n+1$ positions of insertion ``positions'' or ``gaps''. Let us make some definitions to keep track of these statistics $T_1,T_2,T_3$ more closely.

\begin{enumerate}
    \item $\cl{A}(p) := |\{ i \in [n-1] : p_i, p_{i+1} \in \mathbb{E} \} \cup \{ 0 : p_1 \in \mathbb{E} \}|.$
    \item $\cl{B}(p) := |\{ i \in [n-1] : p_i > p_{i+1} \mbox{ and } p_i \in \mathbb{O} \} \cup \{ n : p_n \in \mathbb{O} \}|.$
    \item $\cl{C}(p) := |\{ i \in [n-1] : (p_i \in \mathbb{E} \mbox{ and } p_{i+1} \in \mathbb{O}) \mbox{ or } (p_i < p_{i+1} \mbox{ and } p_i, p_{i+1} \in \mathbb{O}) \} \cup \{ 0 : p_1 \in \mathbb{O} \}|.$
    \item $\cl{D}(p) := |\{ i \in [n-1] : p_i < p_{i+1}, \ p_i \in \mathbb{O}, \ p_{i+1} \in \mathbb{E} \}|.$
    \item $\cl{E}(p) := |\{ n : p_n \in \mathbb{E} \}|.$
\end{enumerate}
 We classify these gaps into types according to the statistics defined above. Specifically, an $\cl{A}$-gap is a gap whose location is exactly one of the positions counted by $\cl{A}(p)$, the gap between an even $p_j$ and an even $p_{j+1}$, or the starting gap if $p_1$ is even. We define $\cl{B}$-gaps, $\cl{C}$-gaps, $\cl{D}$-gaps, and $\cl{E}$-gaps analogously based on the conditions for $\cl{B}(p)$ through $\cl{E}(p)$. Similarly, for $m \in \{1, 2, 3\}$, a $T_m$-gap denotes a gap strictly between adjacent elements $p_j$ and $p_{j+1}$ that satisfies the condition for $T_m(p)$. For instance, a $T_1$-gap is the position between an odd $p_j$ and an even $p_{j+1}$ where $p_j > p_{j+1}$. Whenever we refer to the $k$th gap of a certain type, we assume they are read from left to right.
We construct an involution that witnesses this equidistribution recursively.

\begin{lemma} (Odd insertions)\label{lem: odd-insertion}
Let $n \geq 2$ be an even integer. Let $p, p'\in \SSS_n$ such that $p'=f_n(p)$, where $f_n$ has the properties mentioned in Lemma \ref{lem: even-insertion}. Consider the map $f_{n+1}:\SSS_{n+1}\rightarrow \SSS_{n+1}$ that sends:
\begin{enumerate}
    \item the permutation with $n+1$ inserted in the $k$th $\cl{A}$-gap in $p$ to the permutation with $n+1$ inserted in the $k$th $\cl{B}$-gap in $p'$,
    \item the permutation with $n+1$ inserted in the $k$th $\cl{B}$-gap in $p$ to the permutation with $n+1$ inserted in the $k$th $\cl{A}$-gap in $p'$,
    \item the permutation with $n+1$ inserted in the $k$th $\cl{C}$-gap (respectively, $\cl{D},\cl{E}$-gaps) in $p$ to the permutation with $n+1$ inserted in the $k$th $\cl{C}$-gap (respectively, $\cl{D},\cl{E}$-gaps) in $p'$.
\end{enumerate} 
For every $p_{new} \in \SSS_{n+1}$ and $p_{new}':=f_{n+1}(p_{new})$, we have:
\begin{enumerate}\label{rel: odd}
        \item $\cl{A}(p_{new})=\cl{B}(p_{new}')-1$,
        \item $\cl{B}(p_{new})=\cl{A}(p_{new}')+1$,
        \item $\cl{C}(p_{new})=\cl{C}(p_{new}')$,
        \item $\cl{D}(p_{new})=\cl{D}(p_{new}')$,
        \item $\cl{E}(p_{new})=\cl{E}(p_{new}')$.
    \end{enumerate}
Finally, we have $(T_1(p_{new}),T_2(p_{new}),T_3(p_{new}))=(T_2(p_{new}'),T_1(p_{new}'),T_3(p_{new}'))$.
\end{lemma}

\begin{proof}
    It is easy to verify that inserting $n+1$ into any type gap reduces the number of gaps of that type by $1$ while increasing the number of $\cl{B}$-gaps by $1$ and $\cl{C}$-gaps by $1$. Finally, inserting in an $\cl{A}$-gap increases the number of $T_1$ by $1$ without changing the others while inserting in a $\cl{B}$-gap increases the number of $T_2$ by $1$ without changing the others. The result follows.
\end{proof}

\begin{lemma} (Even insertions)\label{lem: even-insertion}
Let $n \geq 1$ be an odd integer. Let $p, p' \in \SSS_n$ such that $p'=f_n(p)$, where $f_n$ has the properties mentioned in Lemma \ref{lem: odd-insertion}. Consider the map $f_{n+1}:\SSS_{n+1}\rightarrow \SSS_{n+1}$ that sends:
\begin{enumerate}
    \item the permutation with $n+1$ inserted in the $k$th $T_1$-gap in $p$ to the permutation with $n+1$ inserted in the $k$th $T_2$-gap in $p'$,
    \item the permutation with $n+1$ inserted in the $k$th $T_2$-gap in $p$ to the permutation with $n+1$ inserted in the $k$th $T_1$-gap in $p'$,
    \item the permutation with $n+1$ inserted in the first gap in $p$ to the permutation with $n+1$ inserted in the first gap in $p'$,
    \item the permutation with $n+1$ inserted in the $k$th $T_3$-gap in $p$ to the permutation with $n+1$ inserted in the $k$th $T_3$-gap in $p'$,
    \item the permutation with $n+1$ inserted in the last gap in $p$ to the permutation with $n+1$ inserted in the last gap in $p'$,
    \item the permutation with $n+1$ inserted in the $k$th gap that is not one of the above in $p$ to the permutation with $n+1$ inserted in the $k$th gap that is not one of the above in $p'$.
\end{enumerate}
Then, for every $p_{new}\in \SSS_{n+1}$ and $p_{new}':=f_{n+1}(p_{new})$, we have:
\begin{enumerate}\label{rel: even}
        \item $\cl{A}(p_{new})=\cl{B}(p_{new}')$,
        \item $\cl{B}(p_{new})=\cl{A}(p_{new}')$,
        \item $\cl{C}(p_{new})=\cl{C}(p_{new}')$,
        \item $\cl{D}(p_{new})=\cl{D}(p_{new}')$,
        \item $\cl{E}(p_{new})=\cl{E}(p_{new}')$.
    \end{enumerate}
Finally, we have $(T_1(p_{new}),T_2(p_{new}),T_3(p_{new}))=(T_2(p_{new}'),T_1(p_{new}'),T_3(p_{new}'))$.
\end{lemma}
\begin{proof}
Let $n$ be an odd integer, making $n+1$ an even integer and strictly the maximum element in the permutation. From the premise of the map $f_n$ (Lemma \ref{lem: odd-insertion}), we assume the following initial state for $p$ and $p'$:
\begin{enumerate}
    \item $\cl{A}(p) = \cl{B}(p') - 1$
    \item $\cl{B}(p) = \cl{A}(p') + 1$
    \item $\cl{C}(p) = \cl{C}(p')$, $\cl{D}(p) = \cl{D}(p')$, $\cl{E}(p) = \cl{E}(p')$
    \item $(T_1(p), T_2(p), T_3(p)) = (T_2(p'), T_1(p'), T_3(p'))$
\end{enumerate}

We analyse the changes in the statistics case by case.

\noindent \textbf{Case 1:} \\
Suppose $n+1$ is inserted into a $T_1$-gap in $p$ and the corresponding $T_2$-gap in $p'$.
\begin{itemize}
    \item \textbf{Inserting into $T_1$:} A $T_1$-gap implies $x$ is odd, $y$ is even, and $x > y$. This destroys a $T_1$-gap (and thus a $\cl{B}$-gap). We create $(x, n+1)$ with odd $x <$ even $n+1$, which is a $\cl{D}$-gap. We create $(n+1, y)$ with $n+1 >y$ ($y$ even), which is an $\cl{A}$-gap. The net change in $p$ is $\Delta\cl{A}=+1, \Delta\cl{B}=-1, \Delta\cl{D}=+1$.
    \item \textbf{Inserting into $T_2$:} A $T_2$-gap implies $x$ is odd, $y$ is odd, and $x < y$. This destroys a $T_2$-gap (a $\cl{C}$-gap). We create $(x, n+1)$, which is a $\cl{D}$-gap. We create $(n+1, y)$ with $n+1 >y$ ($y$ odd), which is a $\cl{C}$-gap. The net change in $p'$ is $\Delta\cl{C}=0, \Delta\cl{D}=+1$.
    \item \textbf{Verifying Equations:}
    \begin{eqnarray*}
        \cl{A}(p_{new}) &=& \cl{A}(p) + 1 = (\cl{B}(p') - 1) + 1 = \cl{B}(p') = \cl{B}(p'_{new}),\\
        \cl{B}(p_{new}) &=
        &\cl{B}(p) - 1 = (\cl{A}(p') + 1) - 1 = \cl{A}(p') = \cl{A}(p'_{new}).
    \end{eqnarray*}
    $\cl{C}, \cl{D},$ and $\cl{E}$ remain equal in both. By symmetry, the exact reverse holds when $T_2$ in $p$ is mapped to $T_1$ in $p'$.
\end{itemize}

\noindent \textbf{Case 2:} \\
Suppose $n+1$ is inserted into a $T_3$-gap in both $p$ and $p'$.
\begin{itemize}
    \item A $T_3$-gap implies $x$ is odd, $y$ is odd, and $x > y$. This destroys a $T_3$-gap (a $\cl{B}$-gap). We create $(x, n+1)$ which is a $\cl{D}$-gap, and $(n+1, y)$ which is a $\cl{C}$-gap.
    \item The net change for both permutations is $\Delta\cl{B}=-1, \Delta\cl{C}=+1, \Delta\cl{D}=+1$.
    \item \textbf{Verifying Equations:} \\
    Applying $\Delta\cl{B}=-1$ to both sides of the initial equations yields $\cl{A}(p_{new})=\cl{B}(p'_{new})$ and $\cl{B}(p_{new})=\cl{A}(p'_{new})$. $\cl{C}, \cl{D},$ and $\cl{E}$ change in the same way for both.
\end{itemize}

\noindent \textbf{Case 3:} \\
Suppose $n+1$ is inserted at the end of both permutations. Let the original last element be $x$. Because $\cl{E}(p) = \cl{E}(p')$, both permutations terminate with an element of the same parity.
\begin{itemize}
    \item \textbf{If $x$ is odd (Odd End):} Removes a $\cl{B}$-gap. Creates $(x, n+1)$, a $\cl{D}$-gap. Create an even end, an $\cl{E}$-gap. The net change for both is $\Delta\cl{B}=-1, \Delta\cl{D}=+1, \Delta\cl{E}=+1$. As in Case 2, dropping $\cl{B}$ by $1$ on both sides perfectly restores cross-symmetry.
    \item \textbf{If $x$ is even (Even End):} Removes an $\cl{E}$-gap. Creates $(x, n+1)$, an $\cl{A}$-gap. Create an even end, an $\cl{E}$-gap. The net change for both is $\Delta\cl{A}=+1, \Delta\cl{E}=0$.
    \item \textbf{Verifying Equations:}
    \begin{eqnarray*}
        \cl{A}(p_{new}) &=& \cl{A}(p) + 1 = \cl{B}(p') = \cl{B}(p'_{new}),\\
        \cl{B}(p_{new}) &=& \cl{B}(p) = \cl{A}(p') + 1 = \cl{A}(p'_{new}).
    \end{eqnarray*}
\end{itemize}

\noindent \textbf{Case 4:} \\
We evaluate the insertion of an even $n+1$ into the first gap, Even-Even gaps, Even-Odd gaps, and $\cl{D}$-gaps (odd-even ascents). This will cover all the remaining cases.
\begin{itemize}
    \item \textbf{First gap:} If $y$ is odd (removes a $\cl{C}$, creates a $\cl{A}$ start, creates a $\cl{C}$). If $y$ is even (removes a $\cl{A}$, creates a $\cl{A}$ start, creates a $\cl{A}$). The net change is universally $\Delta\cl{A}=+1$.
    \item \textbf{Even-Even:} Removes a $\cl{A}$, creates two $\cl{A}$-gaps. Net change: $\Delta\cl{A}=+1$.
    \item \textbf{Even-Odd:} Removes a $\cl{C}$, creates $\cl{A}$ and $\cl{C}$. Net change: $\Delta\cl{A}=+1$.
    \item \textbf{$\cl{D}$-gap:} Removes a $\cl{D}$, creates $\cl{D}$ and $\cl{A}$. Net change: $\Delta\cl{A}=+1$.
\end{itemize}
In both permutations, a strict net change of $\Delta\cl{A}=+1$ occurs. Applying this to both sides of the initial equations gives $\cl{A}(p_{new})=\cl{B}(p'_{new})$ and $\cl{B}(p_{new})=\cl{A}(p'_{new})$.

It is easy to verify that $(T_1(p_{new}),T_2(p_{new}),T_3(p_{new}))=(T_2(p_{new}'),T_1(p_{new}'),T_3(p_{new}'))$ occurs.
\end{proof}
We are now equipped to prove Theorem \ref{thm: main}.
\begin{proof}[Proof of Theorem \ref{thm: main}]
    We prove the stronger claim that there exist maps $f_n:\SSS_n\rightarrow \SSS_n$ with $$(T_1(p),T_2(p),T_3(p),S_{17}(p))=(T_2(p'),T_1(p'),T_3(p'),S_{17}(p'))$$ for each $p\in \SSS_n$ and $p':=f_n(p)$ such that relations in \ref{rel: odd} are satisfied when $n$ is odd and the relations in \ref{rel: even} are satisfied when $n$ is even. The proof is by induction.
    Clearly, $f_1$ is well defined by sending $1$ to itself and $\cl{A}(1),\cl{B}(1),\cl{C}(1),\cl{D}(1),\cl{E}(1)$ are $0,1,1,0,0$ respectively. This satisfies the relations mentioned in \ref{rel: odd}. This gives us our base case.
    
    If we assume that the map exists for some $n_0$, then Lemmas \ref{lem: odd-insertion},\ref{lem: even-insertion} guarantee the map $f_{n_0+1}$ exists  and that it is an involution, which completes our induction. 
\end{proof}
We now give the proof of Corollary \ref{cor: S10_S12}.

\begin{proof}[Proof of Corollary \ref{cor: S10_S12}]
    Observe from our earlier definitions that $S_{10} = T_1 + T_3$ and $S_{12} = T_2 + T_3$. Since the quadruples $(T_1, T_2, T_3, S_{17})$ and $(T_2, T_1, T_3, S_{17})$ are equidistributed over $\SSS_n$ by Theorem \ref{thm: main}, the result follows immediately.
\end{proof}
\subsection{Examples}
\begin{example}
We show an example of the involution for the permutation $p=36457821$. Table \ref{tab:bijection_trace} shows the insertion process. 
    \begin{table}[H]
    \centering
    \renewcommand{\arraystretch}{1.3}

    \setlength{\tabcolsep}{4pt}
    \begin{adjustbox}{width=\textwidth}
    \begin{tabular}{c l ccccc l ccccc}
        \hline
        $n$ & $p$ & $\cl{A}(p)$ & $\cl{B}(p)$ & $\cl{C}(p)$ & $\cl{D}(p)$ & $\cl{E}(p)$ & $p'$ & $\cl{A}(p')$ & $\cl{B}(p')$ & $\cl{C}(p')$ & $\cl{D}(p')$ & $\cl{E}(p')$ \\
        \hline
        1 & 1 & 0 & 1 & 1 & 0 & 0 & 1 & 0 & 1 & 1 & 0 & 0 \\
        2 & 2 1 & 1 & 1 & 1 & 0 & 0 & 2 1 & 1 & 1 & 1 & 0 & 0 \\
        3 & 3 2 1 & 0 & 2 & 2 & 0 & 0 & 2 1 3 & 1 & 1 & 2 & 0 & 0 \\
        4 & 3 4 2 1 & 1 & 1 & 2 & 1 & 0 & 2 1 4 3 & 1 & 1 & 2 & 1 & 0 \\
        5 & 3 4 5 2 1 & 0 & 2 & 3 & 1 & 0 & 2 1 4 3 5 & 1 & 1 & 3 & 1 & 0 \\
        6 & 3 6 4 5 2 1 & 1 & 2 & 3 & 1 & 0 & 2 6 1 4 3 5 & 2 & 1 & 3 & 1 & 0 \\
        7 & 3 6 4 5 7 2 1 & 1 & 2 & 4 & 1 & 0 & 7 2 6 1 4 3 5 & 1 & 2 & 4 & 1 & 0 \\
        8 & 3 6 4 5 7 8 2 1 & 2 & 1 & 4 & 2 & 0 & 7 2 6 1 4 3 8 5 & 1 & 2 & 4 & 2 & 0 \\
        \hline
    \end{tabular}
    \end{adjustbox}
    \vspace{0.2cm}
    \caption{Step-by-step insertions showing $p$ and its image $p'$ under the involution, along with their respective statistics.}
    \label{tab:bijection_trace}
\end{table}
\end{example}
\begin{example}
    We show another example of the involution for the permutation $p=14356728$. Table \ref{tab:bijection_trace_2} shows the insertion process for the permutation.
    \begin{table}[H]
    \centering
    \renewcommand{\arraystretch}{1.3}
    \setlength{\tabcolsep}{3.5pt}
    \begin{adjustbox}{width=\textwidth}
    \begin{tabular}{c l ccccc l ccccc}
        \hline
        $n$ & $p$ & $\cl{A}(p)$ & $\cl{B}(p)$ & $\cl{C}(p)$ & $\cl{D}(p)$ & $\cl{E}(p)$ & $p'$ & $\cl{A}(p')$ & $\cl{B}(p')$ & $\cl{C}(p')$ & $\cl{D}(p')$ & $\cl{E}(p')$ \\
        \hline
        1 & 1 & 0 & 1 & 1 & 0 & 0 & 1 & 0 & 1 & 1 & 0 & 0 \\
        2 & 1 2 & 0 & 0 & 1 & 1 & 1 & 1 2 & 0 & 0 & 1 & 1 & 1 \\
        3 & 1 3 2 & 0 & 1 & 2 & 0 & 1 & 1 3 2 & 0 & 1 & 2 & 0 & 1 \\
        4 & 1 4 3 2 & 0 & 1 & 2 & 1 & 1 & 1 3 4 2 & 1 & 0 & 2 & 1 & 1 \\
        5 & 1 4 3 5 2 & 0 & 1 & 3 & 1 & 1 & 1 3 4 5 2 & 0 & 1 & 3 & 1 & 1 \\
        6 & 1 4 3 5 6 2 & 1 & 0 & 3 & 2 & 1 & 1 6 3 4 5 2 & 0 & 1 & 3 & 2 & 1 \\
        7 & 1 4 3 5 6 7 2 & 0 & 1 & 4 & 2 & 1 & 1 6 3 4 5 7 2 & 0 & 1 & 4 & 2 & 1 \\
        8 & 1 4 3 5 6 7 2 8 & 1 & 1 & 4 & 2 & 1 & 1 6 3 4 5 7 2 8 & 1 & 1 & 4 & 2 & 1 \\
        \hline
    \end{tabular}
    \end{adjustbox}
    \vspace{0.2cm}
    \caption{The step-by-step insertion for $p= 14356728$ and $p'=16345728$.}
    \label{tab:bijection_trace_2}
\end{table}
\end{example}
\begin{example}
 Table \ref{tab:bijection_trace_identity} shows the insertion process for the identity permutation of length $8$.
    \begin{table}[H]
    \centering
    \renewcommand{\arraystretch}{1.3}
    \setlength{\tabcolsep}{4pt}
     \begin{adjustbox}{width=\textwidth}
    \begin{tabular}{c l ccccc l ccccc}
        \hline
        $n$ & $p$ & $\cl{A}(p)$ & $\cl{B}(p)$ & $\cl{C}(p)$ & $\cl{D}(p)$ & $\cl{E}(p)$ & $p'$ & $\cl{A}(p')$ & $\cl{B}(p')$ & $\cl{C}(p')$ & $\cl{D}(p')$ & $\cl{E}(p')$ \\
        \hline
        1 & 1 & 0 & 1 & 1 & 0 & 0 & 1 & 0 & 1 & 1 & 0 & 0 \\
        2 & 1 2 & 0 & 0 & 1 & 1 & 1 & 1 2 & 0 & 0 & 1 & 1 & 1 \\
        3 & 1 2 3 & 0 & 1 & 2 & 1 & 0 & 1 2 3 & 0 & 1 & 2 & 1 & 0 \\
        4 & 1 2 3 4 & 0 & 0 & 2 & 2 & 1 & 1 2 3 4 & 0 & 0 & 2 & 2 & 1 \\
        5 & 1 2 3 4 5 & 0 & 1 & 3 & 2 & 0 & 1 2 3 4 5 & 0 & 1 & 3 & 2 & 0 \\
        6 & 1 2 3 4 5 6 & 0 & 0 & 3 & 3 & 1 & 1 2 3 4 5 6 & 0 & 0 & 3 & 3 & 1 \\
        7 & 1 2 3 4 5 6 7 & 0 & 1 & 4 & 3 & 0 & 1 2 3 4 5 6 7 & 0 & 1 & 4 & 3 & 0 \\
        8 & 1 2 3 4 5 6 7 8 & 0 & 0 & 4 & 4 & 1 & 1 2 3 4 5 6 7 8 & 0 & 0 & 4 & 4 & 1 \\
        \hline
    \end{tabular}
    \end{adjustbox}
    \vspace{0.2cm}
    \caption{Step-by-step insertions for the identity permutation, which is a fixed point under the involution.}
    \label{tab:bijection_trace_identity}
\end{table}
\end{example}

\section*{Acknowledgements}

This project was initiated during the ``NCM Workshop: Topics in Geometric Combinatorics'' held at CMI, Chennai in July 2025. The authors are grateful to the organizers for creating a productive and inspiring research environment.


\bibliography{file}

@incollection {stanley-post-coxeter,
    AUTHOR = {Postnikov, Alexander and Stanley, Richard P.},
     TITLE = {Deformations of {C}oxeter hyperplane arrangements},
      NOTE = {In memory of Gian-Carlo Rota},
   JOURNAL = {J. Combin. Theory Ser. A},
  FJOURNAL = {Journal of Combinatorial Theory. Series A},
    VOLUME = {91},
      YEAR = {2000},
    NUMBER = {1-2},
     PAGES = {544--597},
      ISSN = {0097-3165,1096-0899},
   MRCLASS = {52C35},
  MRNUMBER = {1780038},
       DOI = {10.1006/jcta.2000.3106},
       URL = {https://doi.org/10.1006/jcta.2000.3106},
}

@article {stanley-hyperplane-interval-orders,
    AUTHOR = {Stanley, Richard P.},
     TITLE = {Hyperplane arrangements, interval orders, and trees},
   JOURNAL = {Proc. Nat. Acad. Sci. U.S.A.},
  FJOURNAL = {Proceedings of the National Academy of Sciences of the United
              States of America},
    VOLUME = {93},
      YEAR = {1996},
    NUMBER = {6},
     PAGES = {2620--2625},
      ISSN = {0027-8424},
   MRCLASS = {52B30 (05C30 06A06 20F55)},
  MRNUMBER = {1379568},
MRREVIEWER = {Hiroaki\ Terao},
       DOI = {10.1073/pnas.93.6.2620},
       URL = {https://doi.org/10.1073/pnas.93.6.2620},
}

@article {athana-linusson-bij-shi-regions,
    AUTHOR = {Athanasiadis, Christos A. and Linusson, Svante},
     TITLE = {A simple bijection for the regions of the {S}hi arrangement of
              hyperplanes},
   JOURNAL = {Discrete Math.},
  FJOURNAL = {Discrete Mathematics},
    VOLUME = {204},
      YEAR = {1999},
    NUMBER = {1-3},
     PAGES = {27--39},
      ISSN = {0012-365X,1872-681X},
   MRCLASS = {52C35},
  MRNUMBER = {1691861},
MRREVIEWER = {Johann\ Linhart},
       DOI = {10.1016/S0012-365X(98)00365-3},
       URL = {https://doi.org/10.1016/S0012-365X(98)00365-3},
}

@article {hetyei-acyclic-tourney,
    AUTHOR = {Hetyei, G\'abor},
     TITLE = {Alternation acyclic tournaments},
   JOURNAL = {European J. Combin.},
  FJOURNAL = {European Journal of Combinatorics},
    VOLUME = {81},
      YEAR = {2019},
     PAGES = {1--21},
      ISSN = {0195-6698,1095-9971},
   MRCLASS = {05A15 (05A05 05A19 05C20 05C30)},
  MRNUMBER = {3949635},
MRREVIEWER = {Wayne\ M.\ Dymacek},
       DOI = {10.1016/j.ejc.2019.04.007},
       URL = {https://doi.org/10.1016/j.ejc.2019.04.007},
}

@incollection {stanley-hyperplane-notes,
    AUTHOR = {Stanley, Richard P.},
     TITLE = {An introduction to hyperplane arrangements},
 BOOKTITLE = {Geometric combinatorics},
    SERIES = {IAS/Park City Math. Ser.},
    VOLUME = {13},
     PAGES = {389--496},
 PUBLISHER = {Amer. Math. Soc., Providence, RI},
      YEAR = {2007},
      ISBN = {978-0-8218-3736-8; 0-8218-3736-2},
   MRCLASS = {52C35 (05B35 55R80)},
  MRNUMBER = {2383131},
       DOI = {10.1090/pcms/013/08},
       URL = {https://doi.org/10.1090/pcms/013/08},
}

@article {kreweras-genocchi,
    AUTHOR = {Kreweras, G.},
     TITLE = {Sur les permutations compt\'ees par les nombres de {G}enocchi
              de 1-i\`ere et 2-i\`eme esp\`ece},
   JOURNAL = {European J. Combin.},
  FJOURNAL = {European Journal of Combinatorics},
    VOLUME = {18},
      YEAR = {1997},
    NUMBER = {1},
     PAGES = {49--58},
      ISSN = {0195-6698,1095-9971},
   MRCLASS = {05A15 (11B68)},
  MRNUMBER = {1427604},
MRREVIEWER = {Volker\ Strehl},
       DOI = {10.1006/eujc.1995.0081},
       URL = {https://doi.org/10.1006/eujc.1995.0081},
}

@article {equi-desc-ap-pv-deutsch-kitaev,
    AUTHOR = {Deutsch, Emeric and Kitaev, Sergey and Remmel, Jeffrey},
     TITLE = {Equidistribution of descents, adjacent pairs, and place-value
              pairs on permutations},
   JOURNAL = {J. Integer Seq.},
  FJOURNAL = {Journal of Integer Sequences},
    VOLUME = {12},
      YEAR = {2009},
    NUMBER = {5},
     PAGES = {Article 09.5.1, 19},
      ISSN = {1530-7638},
   MRCLASS = {05A05},
  MRNUMBER = {2520840},
}

@article {kitaev-remmel-parity,
    AUTHOR = {Kitaev, Sergey and Remmel, Jeffrey},
     TITLE = {Classifying descents according to parity},
   JOURNAL = {Ann. Comb.},
  FJOURNAL = {Annals of Combinatorics},
    VOLUME = {11},
      YEAR = {2007},
    NUMBER = {2},
     PAGES = {173--193},
      ISSN = {0218-0006,0219-3094},
   MRCLASS = {05A15},
  MRNUMBER = {2336014},
MRREVIEWER = {Marc\ Noy},
       DOI = {10.1007/s00026-007-0313-2},
       URL = {https://doi.org/10.1007/s00026-007-0313-2},
}

@article {lazar-wachs-hom-linial,
    AUTHOR = {Lazar, Alexander and Wachs, Michelle L.},
     TITLE = {The homogenized {L}inial arrangement and {G}enocchi numbers},
   JOURNAL = {Comb. Theory},
  FJOURNAL = {Combinatorial Theory},
    VOLUME = {2},
      YEAR = {2022},
    NUMBER = {1},
     PAGES = {Paper No. 2, 34},
      ISSN = {2766-1334},
   MRCLASS = {52C35 (05A05 05A15 05B35 06A07 11B68)},
  MRNUMBER = {4405991},
MRREVIEWER = {Piotr\ Pokora},
       DOI = {10.5070/c62156874},
       URL = {https://doi.org/10.5070/c62156874},
}

@article {whitney-logical-exp, 
    AUTHOR = {Whitney, Hassler},
     TITLE = {A logical expansion in mathematics},
   JOURNAL = {Bull. Amer. Math. Soc.},
  FJOURNAL = {Bulletin of the American Mathematical Society},
    VOLUME = {38},
      YEAR = {1932},
    NUMBER = {8},
     PAGES = {572--579},
      ISSN = {0002-9904},
   MRCLASS = {99-04},
  MRNUMBER = {1562461},
       DOI = {10.1090/S0002-9904-1932-05460-X},
       URL = {https://doi.org/10.1090/S0002-9904-1932-05460-X},
}

@article {zaslavsky-facing-arrangments,
    AUTHOR = {Zaslavsky, Thomas},
     TITLE = {Facing up to arrangements: face-count formulas for partitions
              of space by hyperplanes},
   JOURNAL = {Mem. Amer. Math. Soc.},
  FJOURNAL = {Memoirs of the American Mathematical Society},
    VOLUME = {1},
      YEAR = {1975},
     PAGES = {vii+102},
      ISSN = {0065-9266,1947-6221},
   MRCLASS = {05A15},
  MRNUMBER = {357135},
MRREVIEWER = {E.\ Jucovi\v c},
       DOI = {10.1090/memo/0154},
       URL = {https://doi.org/10.1090/memo/0154},
}
\bibliographystyle{acm}
\end{document}